\documentclass[11pt]{article}
\usepackage{graphicx}
\usepackage{amsmath}
\usepackage{mathrsfs,amsfonts,bbm}
\usepackage{amsfonts}
\usepackage{amsmath,amssymb,amsfonts}
\usepackage{amssymb}
\usepackage{soul} 
\usepackage{color}
\definecolor{c}{rgb}{0.9,0.3,0.1}
\definecolor{b}{rgb}{0.1,0.3,0.9}

\usepackage{ulem}

 \allowdisplaybreaks

\newtheorem{theorem}{Theorem}[section]

\newtheorem{lemma}[theorem]{Lemma}
\newtheorem{prop}[theorem]{Proposition}
\newtheorem{definition}[theorem]{Definition}

\newtheorem{remark}[theorem]{Remark}

\newtheorem{hypothesis}[theorem]{Hypothesis}

\renewcommand{\theequation}{\arabic{section}.\arabic{equation}}

\numberwithin{equation}{section}

\def\be{{\beta}}
\def\de{{\delta}}

\def\ka{{\kappa}}

\def\si{{\sigma}}

\def\va{{\varepsilon}}

\def\<{\langle}
\def\>{\rangle}
\def\({\left(}
\def\){\right)}

\newfam\msbmfam\font\tenmsbm=msbm10\textfont
\msbmfam=\tenmsbm\font\sevenmsbm=msbm7
\scriptfont\msbmfam=\sevenmsbm\def\bb#1{{\fam\msbmfam #1}}

\def\RR{\bb R}

\newcommand{\E}{{\mathbbm{E}}}

\def\cM{{\cal M}}
 
\def\cU{{\cal U}}

\newcommand{\EE}{{\mathbb{E}}}

\newcommand{\R}{{\mathbf{R}}}

 \def\qed{\hfill $\Box$}

\def\sC{\mathcal C}
\def\sF{\mathcal F}
\def\sG{\mathcal G}
\def\sH{\mathcal H}

\def\sM{{\mathcal M}}
\def\sN{{\mathcal N}}
\def\sP{{\mathcal P}}
\def\sW{{\mathcal W}}

\def\wt{\widetilde}
\def\wh{\widehat}

\def\eps{\varepsilon}

\def\R{\mathbbm{R}}
\def\bF{\mathbbm{F}}
\def\bP{\mathbbm{P}}

\def\1{\mathbbm{1}}

\def\bP{\mathbbm P}
\def\pr{{\prime}}

\def\pf{\noindent {\bf Proof.} }

\numberwithin{equation}{section}

\begin{document}
\title{Mean-field stochastic differential equations driven by sub-diffusions and their control problem}
  \author{ Shuaiqi Zhang   \footnote{Research supported by the Humanities and Social Sciences Foundation of Ministry of Education of
China (Application of Anomalous Sub-diffusion in Insurance and Finance, Grant No. 24YJA910008)    and National Natural Science Foundation of China (Grant No. 12571519).}
\qquad \hbox{and} \qquad 
 { Zhen-Qing Chen   \footnote{Research  partially supported by a Simons Foundation fund.   } 
 } 
} \maketitle
\date
\begin{abstract}  
In this paper,  we establish the existence and uniqueness of solutions for mean-field   stochastic differential equations (MF-SDEs in short)  and  backward stochastic differential equations  (MF-BSDE) driven by  anomalous sub-diffusions $\{B_{L_t}; t\geq 0\}$ with random coefficients, respectively.  Here $B$ is a Brownian motion on $\R^d$ and $L$ is  the inverse of  a subordinator  $S$ with drift 
$\kappa >0$ that is independent of $B$.  We further study the stochastic maximum principles (SMPs) for  control problems  of the stochastic systems modelled by the   MF-SDEs   using a convex variational  method.  A linear quadratic control example is given  in the last section of this paper,  for which both the SMP and the sufficient SMP established in this paper are utilized to show explicitly that it admits a unique stochastic optimal control.
 \end{abstract}

   \medskip
 

\medskip


\section{Introduction}\label{S1}

Mean-field theory  is a powerful tool for  studying  collective behavior arising from  mutual interactions.
It has been applied to a wide range of fields including  physics,  statistical inference, graphical models, neuroscience, 
artificial intelligence, epidemic models, computer-network performance and game theory. 
Traditionally, most existing literature is within the framework of mean-field models driven by Brownian motion. 
 These mean-field stochastic differential equations (SDEs) can be traced back to the McKean–Vlasov model, originally introduced by Kac \cite{Kac2007} and McKean \cite{McKean1966} to investigate physical systems with large numbers of interacting particles. 
Regarding applications, Lasry and Lions \cite{Lasry2007} extended mean-field models to economics and finance by considering $N$-player stochastic differential games, proving the existence of Nash equilibria and rigorously deriving the mean-field limit equations as $N \to \infty$. A natural question then arises: how should the system be modeled when the underlying process is less active or exhibits memory effects?

Sub-diffusions are a class of random processes describing particle motion that is slower than Brownian motion, often due to particle sticking or trapping. Such phenomena have been observed in diverse real-world systems, including porous media, biological systems, and financial markets. For instance, sub-diffusion effectively models disease spread in specific populations and particle transport through soil. 
Prototypes of anomalous sub-diffusions are Brownian motions time-changed by the inverse of subordinators 
that are independent of the Brownian motions (see, e.g., \cite{Meerschaert2004, MK}),
which notably lack the Markov property. 
This inspires us to propose a mean-field framework driven by sub-diffusion to model state processes characterized by collective behavior in a slowly evolving   random 
environment and study its well-posedness. A practical example is a "bear market," where trading is less active and stock prices, influenced by collective investor sentiment, are characterized by their ensemble average. This further motivates us to explore the  optimal control for mean-field systems driven by sub-diffusion. Given the theoretical and practical significance of such systems, it is compelling to investigate 
 mean-field   SDEs and   its associated control problems.

As for the   mean-field control   in the classical Brownian motion setting,  various versions of the stochastic maximum principle (SMP) have been developed   in 
 different frameworks (see, e.g., \cite{Acciaio, Buckdahn2016, Lakhdari2021, Li2012, Wang2022, Zhang2018}). 
  To the best of the authors' knowledge, this is the first work to  investigate      MF- SDEs driven by non-Markovian anomalous sub-diffusions.
While sub-diffusions have Brownian motion as their extreme case, 
 the main difference between  SDEs and BSDEs  driven by Brownian motion and by sub-diffusions is that the latter ones 
  can be degenerate in the sense that sub-diffusions can remain constant (i.e., inactive) during infinitely many random time intervals.

\medskip

 Suppose that $S=\{S_t; t\geq 0\}$ is a subordinator with drift $\kappa > 0$ and  
L\'evy measure $\nu$; that is, $S_t=\kappa t+ S_t^0$, where $S_t^0$ is a driftless subordinator with
  L\'evy measure $\nu$.
 Let $L:= \{L_t; t\geq 0\}$ be the inverse of $S$, that is, 
 $$
 L_t = \inf\{r>0: S_r >t\} \quad \hbox{for } t\geq 0.
 $$
 The inverse subordinator $L=\{L_t; t\geq 0\}$ is continuous in $t$  but stays constant during infinitely many time periods which are resulted from the infinitely many jumps by the subordinator $S $ during the entire time interval $[0, \infty)$ when its L\'evy measure is non-trivial.  
Let $B$ be a  Brownian motion that is independent of  the subordinator $S $. 
  Its time-change process $\{B_{L_t}; t\geq 0\}$ by the inverse subordinator $L $ is a typical example of  sub-diffusions, 
which  is a continuous martingale with quadratic variation
$\langle B_{L} \rangle_t  = L_t$  but is not a Markov process. 
   Note that  $B_{L_t}$ stays flat during the time periods when $L_t$ stays constant. 
   For any $\eps >0$,  the jumps of $S_t$ of size larger than $\eps$ occurs 
   according to a Poisson process with parameter $\nu (\eps, \infty)$.  When the L\'evy measure $\nu$ of 
   the subordinator $S$ is infinite, then  during any finite time intervals, 
   $L_t$ has infinitely many small time periods  
      but  only finite many time intervals with length larger than $\eps$  during which it stays constants. 
   Thus the sub-diffusion $B_{L_t}$  matches well with the phenomena such as the financial market constantly  has small corrections 
    but long bear market occurs only sporadically.

 \medskip

   Since $B_{L_t}$ stays flat during the time periods when $L_t$ stays constant, 
  the   MF- SDEs driven by sub-diffusions exhibit  a combined deterministic and stochastic features.  
  The root of this phenomenon is best exemplified by \eqref{e:2.2} below when the subordinator $S$ has positive drift $\kappa >0$. 
  Note that when the L\'evy measure $\nu$ for the subordinator  vanishes and $\kappa =1$,  $S_t=t$ and so 
$B_{L_t}= B_t$ reduces to the standard Brownian motion. Thus 
the results in this paper  not only  recover but also extend in a ``continuous way" the corresponding results in
the classical Brownian setting.

 \medskip
 
 The main results of this paper are Theorems \ref{exis-unique-SDEs},  \ref{T:BSDE},  
   \ref{Necessaryconv} and \ref{Sufficiency}.
   The main contributions and novel aspects of this paper are as follows:
 
 \begin{enumerate}
\item[\rm (i)]  
We establish in Theorems  \ref{exis-unique-SDEs} and  \ref{T:BSDE} the existence and uniqueness 
for   MF-SDEs and  MF-BSDEs driven by anomalous sub-diffusions   with random coefficients.  
 These SDEs are run on two time scales: the Lebesgue time scale $dt$ and the stochastic inverse subordinator scale $dL_t$.  
In addition to the $dL_t$ driver term,  the driver term for $dt$ is also allowed to be dependent on $Z_t$ 
for BSDE. This is new even in the non-distribution dependent BSDEs studied recently in \cite{ZhangChen2024SICON1, ZhangChen2024JDE}.

\item[\rm (ii)] 
The model, mean-field SDEs driven by sub-diffusions,   bridges the gap between microscopic 
dynamics and macroscopic population-level descriptions 
 for  stochastic systems characterized  by   state-variable dynamics whose mean-square displacement grows sub-linearly in time.
 Such anomalous  behavior reflects non-Brownian transport arising in real-world settings through mechanisms 
 such as trapping, heterogeneous environments, or long-tailed waiting times between successive transitions.

\item[\rm (iii)] 
The control of  MF-SDEs reduces the high-dimensional challenge of coordinating a large number of agents to a manageable system involving a single macroscopic density and a centralized control law. This framework is essential for designing policies that optimize the collective behavior of agents without requiring individual tracking.   Thus it is important both in theory and in applications to establish the corresponding control theory for mean-field SDEs driven by sub-diffusions. 
This paper represent the first exploration in this direction.  
  
\item[\rm (iv)] It is shown in this paper that  control problems for the mean-filed mean-field SDEs   driven by sub-diffusions 
have  the distinct combined feature of deterministic  and stochastic characteristics; 
  see, for instance, Theorems  \ref{Necessaryconv} and \ref{Sufficiency}.  
  \end{enumerate}

   Throughout  this paper, we use notation $:=$ as a way of definition. 
  For a stochastic process $x$, we use the notation $x(t)$ and $x_t$ interchangeably, 
  to denote its state or position  at time $t$.   When there is no danger of confusion, 
  for notational simplicity, sometimes we also use $x(t)$ or $x_t$ to denote the stochastic process $x$. 
 
 \medskip
 
 Unless otherwise stated, all vectors in this paper are column vectors. We identify $n\times d$-matrices with points in 
 $\R^{n\times d}$. We use $\xi^*$ to denote the transpose of  a vector or a matrix $\xi$. 
 For two vectors $\xi, \eta \in \R^n$, $\xi \cdot \eta$, or simply $\xi \eta$, denotes their inner product. 
 For a real-valued differentiable function $f(x)$ on $\R^n$, its gradient $\nabla f(x)$ is the column vector
 $(\frac{\partial f(x)}{\partial x_1} , \cdots, \frac{\partial f(x)}{\partial x_n})^*$.  
 For an $\R^N$-valued function $b(x)$ on $\R^n$, $\nabla_x b(x)$ is the $n\times N$-matrix-valued function on $\R^n$
  so that 
 $$
  (\nabla_x b(x))  \xi = \nabla_x (b(x) \cdot \xi) \quad \hbox{for any } \xi \in \R^N.
  $$
  For an $N\times d$-matrix-valued function $\sigma (x)= (\sigma_{ij}(x)) $ on $\R^n$,  
  $\nabla_x \sigma (x):= ( \nabla_x \sigma_{ij}) (x)$ is the 
  $(n\times N) \times d$-matrix-valued function on $\R^n$, which has the property that
  $$
 ( \nabla_x \sigma (x) \xi )  \eta = \nabla_x ( (\sigma (x) \xi) \cdot \eta)
 \quad \hbox{for any } \xi \in \R^d, \eta \in \R^N.
  $$
 For two $n\times d$-matrices $A_1$ and $A_2$,  we use $A_1\cdot A_2$ or ${\rm Tr} (A_1^*A_2)$
  to denote the trace of $(A_1^* A_2)$. Observe that when $A_1$ and $A_2$ are identified with elements in $\R^{n\times d}$,
  ${\rm Tr} (A_1^*A_2)$ is the same as their inner product in $\R^{n\times d}$. 
  
 \medskip
 
 The rest of this paper is organized as follows. In Section \ref{S:2}, we recall some facts from \cite{ZhangChen2024SICON1,  ZC4} 
 about  inverse subordinators and anomalous sub-diffusions that will be used in this paper. 
 In  Section \ref{S:3} and  \ref{S:4},  the existence and uniqueness of the solution to   
  MF-SDEs and  MF-BSDE with random coefficients are established.
  Stochastic control problem is formulate in Section \ref{S:5}.
  We   study  the stochastic maximum principle for 
   systems modeled by  MF-SDEs in Section \ref{S:6},
  A sufficient condition for optimal control is obtained in Section \ref{S:7}. 
  To illustrate the main results of this paper, 
a linear quadratic control example is given  in Section \ref{S:8}, 
 for which we use both the stochastic maximum principle and the sufficient stochastic maximum principle
 to show explicitly that it admits a unique stochastic optimal control.

\section{Preliminary about sub-diffusion}\label{S:2}

In this section, we recall some results from  \cite{ZhangChen2024SICON1,ZC4}
 that will be used later in this paper.
Although the sub-diffusion itself is not a Markov process, we can make it Markov by adding an auxiliary overshoot process. 
 
 \begin{theorem}\label{T:3.1} 
  Suppose that $B$  is a  standard Brownian motion on $\RR^d$  with $d\geq 1$ starting from the origin $\bf 0$, 
   $S$ is any subordinator that is independent of $B$ with $S_0=0$,
and $L_t:=\inf\{r>0:S_r>t\}$.
Then
\begin{equation} \label{e:2.1}
\wt X_t:=(X_t, \, R_t):=
 \left(x_0+ B_{L_{(t-R_0)^+}},  \, R_0+ S_{L_{(t-R_0)^+}} - t  \right), \quad t\geq 0,
\end{equation}
with 
$\wt X_0= (x_0, R_0) \in \RR^d \times [0, \infty)$ 
  is a time-homegenous Markov process taking values in $\R^n\times [0, \infty)$.
\end{theorem}

\bigskip

Note that for a discontinuous subordinator $S$,   $\{S_{L_t}>t\}$ happens with positive probability for each fixed $t>0$.  
 On $\{S_{L_t}>t\}$, the inverse local time $L_s$ and, consequently, 
 the sub-diffusion $B_{L_s}$ remain flat during the time interval $[t, S_{L_t} ]$. 
 We call $R_t:=R_0+ S_{L_{(t-R_0)^+}} - t $ an overshoot process with initial value $R_0$. 
 It measures how much time it would take for the anomalous sub-diffusion $X_t:= x_0+ B_{L_{(t-R_0)^+}}$ to wake up   from time $t$. 
 
\medskip
 
The inverse subordinator $L_t$ is continuous in $t$. 
  Denote by $\{\sG^B_t\}$ the natural augmented filtration generated by the Brownian motion $B$, 
that is, $\sG^B_t :=   \sG_{t+}^{B, 0} \vee {\cal N}$, where $\sG^{B, 0}_t:=\sigma (B_s; s\in [0, t]\}$
 and $\sN:=\{A\in \sG^{B, 0}_\infty: \bP (A)=0\}$. Here the notation $ \sG_{t+}^{B, 0} \vee {\cal N}$ 
  stands for the  $\sigma$-field generated by $ \sG_{t+}^0 \cup  {\cal N}$. 
 The natural augmented filtration $\{\sG_t^B\}$ is right continuous. 
 Similarly, we denote by $\{\sG^L_t\}$ the natural augmented filtration generated by 
 the inverse subordinator $L$.  
 
 \medskip
 
 Fix a constants $T>0 $.   Let  $\wt \sG^B_t := \sG^B_t \vee \sG^L_{T}$  for $t\geq 0$ and 
 \begin{equation}\label{e:2.2a}
  \wt \sF_t:=\wt  \sG^B_{L_t } \quad \hbox{for } 0\leq t\leq T.
 \end{equation} 
 Since $B$ and $L$ are independent, $B$ is an $\{  \wt \sG^B_t\}_{t\geq 0}$-martingale and 
 so  $B_{L_t}$ is a square-integrable $\{\wt \sF_t\}_{0\leq t\leq T}$-martingale. 
  
  \medskip
  
  Denote by $\bF':=\{\sF'_t\}$ the natural augmented filtration generated by the sub-diffusion $B_{L_t}$.
  Clearly, $\sF'_t \subset \wt \sF_t$ for every $t\in [0, T]$.  

\medskip
 
In the rest of this paper, unless otherwise specified,  we assume the subordinator $S$ has positive drift $\kappa >0$. 
In this case,  for any $t, s>0$,
 $$
 0\leq L_{t+s} -L_t \leq s/\kappa.
 $$
 So almost surely, 
\begin{equation}\label{e:2.2}
 \frac{dL_t}{dt}    \ \hbox{ exists for a.e. } t>0  \quad \hbox{with} \quad 0\leq  \frac{dL_t}{dt}  \leq 1/\kappa   \    \hbox{ for a.e. }   t>0.
\end{equation} 
In fact, it is shown in \cite[Proposition 3.2]{ZhangChen2024SICON1} that  $\bP$-a-s.,  
 \begin{equation}\label{e:2.3}
 \frac{dL_t}{dt}   = \kappa^{-1} \1_{\{R_t=0\}} \quad \hbox{for a.e. } t>0, 
\end{equation} 
 where $R_t:=S_{L_t} - t$ is the overshoot process with $R_0=0$. 
  It follows from \eqref{e:2.3} that $\{R_t = 0\} \in \sF'_t$ for every $t\geq 0$.

\medskip

 We will need the following  integral representation of square integrable random variables with respect to  the subdiffusion  $B_{L_t}$, which is crucial for 
 the well-posedness of  MF-BSDEs.
The following result holds for any subordinator $S$; that is, we do not need to assume that $S$ has a positive drift $\kappa>0$. For simplicity, we denote the filtration $\{\wt \sF_s \}_{0\leq s\leq T}$ by $\wt \bF$. 
Note that $\wt \sF_0=\sG^L_T$.

\begin{theorem}[Theroem 2.1 in \cite{ZC4}]\label{T:2.2}
Let $n\geq 1$ be an integer. 
  For every $\R^n$-valued  $\xi \in L^2( \wt \sF_T)$,  
there exists an $\R^{n\times d}$-valued  
$\wt \bF$-progressively measurable process   $\{  H_s;  s\in [0, T]\}$   
 having  $\EE \int_0^T | \wt H_s  |^2 dL_s  <\infty$
so that  
\begin{equation}\label{e:1.1}
\xi  =  \EE \big[ \xi \big| \wt \sF_0\big]  +  \int_0^{T}    H_s    dB_{L_s }  .
\end{equation}
Such $ H$ is unique in the sense that if $\wt H'$ is another $\wt \bF$-progressively measurable process
having  $\EE \int_0^T |  H'_s  |^2 dL_s  <\infty$ so that \eqref{e:1.1} holds, then
$\EE  \int_0^T |  H_s-   H'_s|^2 dL_s =0$.
\end{theorem}

  \medskip

\section{ Existence and uniqueness of solutions for mean-field    SDEs driven by sub-diffusion}\label{S:3}
 
Recall that  $\bF'=\{\sF'_t\}$ is the natural augmented filtration generated by the sub-diffusion $\{B_{L_t}; t\geq 0\}$ on $\R^d$.
Let $n\geq 1$ be an integer, $T>0$ and $x_0\in \R^n$. 
Consider the following   MF-SDEs on $\R^n$ driven by the sub-diffusion $B_{L_t}$ for $t\in [0, T]$: 
  \begin{equation}\label{SDE}
\left\{\begin{aligned}
dX_t =&     b(\omega, t,   X_  t , \bP_{X_t}  )  dt  
+ \delta    (\omega, t,   X_  t , \bP_{X_t} ) dL_t + \sigma(\omega,  t,   X_ t , \bP_{X_t}   )   dB_{L_t} 
 \quad \hbox{for } t\in[0,T],    \\
 X_0=&\ x_0 \in \R^n , 
 \end{aligned}
\right.
 \end{equation}
 where   $\bP _{ X_t }  $  denotes the probability measure (or distribution) induced  by the $\R^n$-valued random variable  $ X_t$. 
For each fixed $x \in \R^n$ and  a probability measure $\mu $ on $\R^n$, 
 $(\omega, t) \rightarrow \varphi (\omega, t, x,  \mu) $ with  $\varphi = b, \delta$ (resp. $\varphi = \sigma $)  is an 
   $\R^n$-valued (resp. $n\times d$-matrix-valued) $\bF'$-progressively measurable random process 
   defined on $\Omega \times  [0, T]$.
 Conditions on $x$ and  $\mu$   for these processes will be imposed later in Hypothesis \ref{HP0-SDEs}.  
For notational simplicity, we will typically drop $\omega$ from the expressions  of the above random processes or variables. 

\medskip

\medskip

Denote by $ \mathcal {P}(\R^n) $  the set of probability measures on $\R^n$.
For $p\geq 1$, define the $p$-Wasserstein's distance
$\mathcal {W}_p (\mu_1,\mu_2) $ between two probability measures $\mu_1, \mu_2 \in  \mathcal {P}(\R^n) $   by 
\begin{eqnarray*} 
 \mathcal {W}_p (\mu_1,\mu_2) := \inf_{\pi \in \sC (\mu_1, \mu_2)}  \left( \int_{\R^n\times\R^n} |x-y| ^p \pi (dx,dy) \right)^{1 / p } ,
\end{eqnarray*}
where $\sC (\mu_1, \mu_2)$ denotes the space of all the probability measures $\pi$ on $\R^n\times \R^n$  
with  marginals $ \mu_1$  and $ \mu_2$. 
 It is easy to see from the definition that
  for any two random variables $(X_1, X_2)$ 
 with  $X_1 \overset{d}=  \mu_1$ and $X_2 \overset{d}=  \mu_2$,   
\begin{equation} \label{e:3.2}
 |\EE[X_1]-\EE[X_2]| \le \mathcal {W}_2 (\mu_1, \mu_2) \le  \(\EE \left[  |X_1-X_2|^2 \right] \)^{1/ 2}.
 \end{equation} 
 Moreover, it is well known that 
   $\mathcal {W}_2 (\mu_1, \mu_2) \geq \mathcal {W}_1 (\mu_1, \mu_2) $. 
  \medskip
 
 \begin{hypothesis}\label{HP0-SDEs} \rm 
\begin{enumerate}
\item[\rm (i)]  
  $\EE \left[ \int_0^T |b (s, {\bf 0}, \delta_{\bf 0})|^2 ds  + \int_0^T \( |\delta  (s, {\bf 0},  \delta_{\bf 0})|^2   
   + |\sigma (s, {\bf 0}, \delta_{\bf 0})|^2\) dL_s  \right] <\infty. $  
	Here  $\bf 0$ denotes the origin in $\R^n$ and  $\delta_{\bf 0}$ denotes the Dirac measure concentrated
	at 0.
	  
 \item[\rm (ii)] $b$,   $\delta$  and $\sigma$  are uniformly Lipschitz continuous in $(x, \mu)$
   with Lipschitz constant $C_0 >0$.
 That is, there is a $\Omega_0\subset \Omega$ with $\bP(\Omega_0)=1$ so that for every
 $\omega \in \Omega_0$, $t\in [0, T]$, $x_i \in \R^n$ and $\mu_i\in \mathcal {P}(\R^n)$
  for $i=1,2$, 
\begin{equation*}
 |\varphi (t, x_1, \mu_1 ) -\varphi (t, x_2, \mu_2)  | 
 \le   C_0 \big(  | x_1-x_2  |	   	+  \mathcal {W}_2(\mu_1, \mu_2) \big)  
\quad \hbox{for } \varphi=b,  \de, \, \si .
  \end{equation*}
  \end{enumerate}
  \end{hypothesis}
 
  \medskip
  
  For $\beta>0$,  define a Banach norm 
$\| \cdot  \|_{ \mathcal M_{\beta}  [0,T]}$  on the space  
  \begin{eqnarray*}
\mathcal M [0,T] 
&:=& \Big\{ \psi(t): \hbox{ $\psi(t)$ is an $\R^n$-valued $\bF'$-progressively  measurable process}  \\
&&  \hskip 0.6 truein 
 \hbox{on $[0, T]$ with } 
  \E   \int_{0}^T  |\psi(t)|^2   dt    <\infty \Big\} 
  \end{eqnarray*}
by 
\begin{eqnarray}\label{x-norm}
\| \psi   \|_{ \mathcal M_{\beta} [0,T]}  
:=  \( \EE\bigg[ \int_{0}^Te^{-\be s}  |\psi(s)|^2ds\bigg]\)^{1/2}  . 
\end{eqnarray}

 \begin{definition}\label{Xsolution}
A stochastic processes $  X  \in \mathcal M [0,T]$ 
is said to be a $L^2$ strong solution of \eqref{SDE}  if  for any $t\in [0, T]$,
   \begin{equation}\label{XSDE}
X_t=x_0 + \int_0^t    b(s, X_s,\bP_{X_s} )ds+  \int_0^t  \delta (s, X_s, \bP_{X_s} )d L_s +  \int_0^t  \sigma   (s, X_s,\bP_{X_s} )  
dB_{L_s}.
\end{equation}
We say the solution to \eqref{SDE} is unique if $ X, \wt X \in \mathcal M[0,T]$  are two $L^2$ strong solutions of \eqref{SDE},
  then $\wt X_t=X_t$   for all $t\in [0, T]$ with probability one.
  \end{definition} 
 
 \begin{theorem}\label{exis-unique-SDEs}
 Suppose that  Hypothesis \ref{HP0-SDEs} holds.
Then for every $x\in \R^n$,  MF-SDE \eqref{SDE} has a unique  $L^2$ strong solution $X $ and 
$\EE \Big[ \sup\limits_{0 \leq t \leq T}  |X_t |^2  \Big]< \infty$ .
 \end{theorem}

\pf Given $\xi= \{ \xi_t; t\in [0, T]\} \in  \mathcal M[0,T]$, let 
   \begin{equation}\label{e:3.5}
X_t =x_0+ \int_0^t   b(s,   \xi_s , \bP_{\xi_s}  )  ds
+\int_0^t \delta (s,   \xi_s , \bP_{\xi_s} ) d L_s  + \int_0^t  \sigma( s,   \xi_s , \bP_{\xi_s}    )   dB_{L_s} 
   \end{equation}
  for $ t\in[0,T] .$ 
   Clearly, $X_t$ is continuous in $t$ and $\sF_t'$-measurable. 
  By  \eqref{e:2.3} and Hypothesis \ref{HP0-SDEs}, there is a constant $C>0$ so that  for every $t\in [0, T]$, 
  \begin{eqnarray*}
  && \E [ |X_t|^2 ] \\
  &\leq & 
  4 |x_0|^2+ 4 \E \left[ \Big| \int_0^t   b(s,   \xi_s , \bP_{\xi_s}  )  ds \Big|^2  +
 \Big| \int_0^t \delta (s,   \xi_s , \bP_{\xi_s} ) d L_s \Big|^2 + 
 \Big| \int_0^t  \sigma( s,   \xi_s , \bP_{\xi_s}    )   dB_{L_s}\Big|^2   \right] \\
 &\leq & 4 |x_0|^2+ 4  \E  \left[ t \int_0^t   | b(s,   \xi_s , \bP_{\xi_s}  ) |^2  ds
  + \kappa^{-1}t  \int_0^t  |\delta (s,   \xi_s , \bP_{\xi_s} )| ^2 d L_s  + 
  \int_0^t  | \sigma( s,   \xi_s , \bP_{\xi_s}    )  |^2   d L_s  \right]\\ 
     &\leq &  4|x_0|^2+  C  \E  \left[ t \int_0^t   | b(s,   0, \delta_0 )  |^2  ds
  +   \int_0^t  | \delta (s,   0, \delta_0 ) |^2 d L_s  + 
  \int_0^t  | \sigma( (s,   0, \delta_0  )  |^2   d L_s  \right]\\ 
  && + C \E \int_0^T \( |\xi_s|^2 + \sW_2 ( \bP_{\xi_s}, \delta_{\bf 0})^2 \) ds  \\
   &\leq &  4|x_0|^2+  C  \E  \left[  \int_0^t   | b(s,   0, \delta_0 )  |^2  ds
  +   \int_0^t  | \delta (s,   0, \delta_0 ) |^2 d L_s  + 
  \int_0^t  | \sigma( (s,   0, \delta_0  )  |^2   d L_s     +  2  \int_0^T |\xi_s|^2  ds  \right] \\
  &< & \infty. 
  \end{eqnarray*}
This shows that $X\in \sM [0, T]$. 
 It defines a map    $\Phi:  \mathcal M [0,T]\rightarrow  \mathcal M  [0,T] $ by sending $\xi \in \sM [0, T] $ to   $  X \in \sM [0, T]$.  
   We next show that the map $\Phi$ is contractive with respect to the Banach norm $\| \cdot \|_{ \sM_\beta [0,T]}$
 for sufficiently large $\beta >1$.  
 
For $\xi$  and $\wt \xi $ in $M[0,T]$, let    $    X = \Phi (\xi )$ and $ \wt X  = \Phi  (\wt \xi )$.
 For notational simplicity, define
 $$
 \wh X_t:= X_t - \wt X_t, \quad  \quad   \wh \xi_t:= \xi_t - \wt \xi_t,    
 $$
 and
 \begin{eqnarray*}
 \wh b (t):=  b(t,   \xi_t, \bP_{\xi_t} )- b(t, \wt \xi_t, \bP_{\wt \xi_t}  ),  \\ \quad
 \wh \delta (t):= \delta(t,   \xi_t, \bP_{\xi_t} )- \delta(t, \wt \xi_t, \bP_{\wt \xi_t}  ), \\
 \quad
   \wh \si (t):= \si (t,   \xi_t, \bP_{\xi_t} )- \si (t,\wt \xi_t, \bP_{\wt \xi_t}  )  .
\end{eqnarray*}
   Then
 $$
  \wh X_t  =  \int_0^t  \wh b(s) ds     +\int_0^t  \wh \delta (s) dL_s   + \int_0^t  \wh \si(s) dB_{L_s}.
$$
 Let $\beta>0$, whose value will be taken to be sufficiently large later. 
By Ito's formula, 
 \begin{eqnarray*}\label{ }
 d \(  e^{-\beta t} |\wh X_t|^2  \)  
 &=&    e^{-\beta t}   \left(-\beta |\wh X_t|^2  +2  \wh X_t \cdot   \wh b ( t)   \right) dt 
 +  e^{-\beta t}\left(2  \wh X_t    \cdot  \wh \delta(t) +|\wh \si(t)|^2 \right)  dL_t \\
 && +     2  e^{-\beta t}      \wh X_t \cdot \wh \si ( t)        d B_{L_t} .     
\end{eqnarray*} 
 Integrating over $[0,T]$ and taking expectation on both sides,  we have by Hypothesis \ref{HP0-SDEs}(ii)
 and \eqref{e:3.2} that 
\begin{eqnarray}
&&  \beta \EE\bigg[\int_{0}^T  e^{-\beta t}   |\wh X_t|^2 dt\bigg]\nonumber\\
&\le&\EE \bigg[\int_{0}  ^T 2 e^{-\beta t}   |\wh X_t|  | \wh b ( t) |     dt\bigg] 
+\EE\bigg[\int_ {0}  ^T     e^{-\beta t}  \(     2  |\wh X_t    |       |  \wh \delta(t)|     +         |\wh \si(t)|^2\)   dL_t \bigg]   \nonumber\\
&\le&   \EE \bigg[\int_{0}^T 2 C_0 e^{-\beta t}   
  |\wh X_t|   \(|\wh \xi_t|  + \mathcal {W}_2  ( \bP_{ \xi_t }, \bP_{  \wt \xi_t  } )   \)      dt\bigg] \nonumber\\
&&+  \EE\bigg[\int_{0}  ^Te^{-\beta t}  \(  2 C_0  |\wh X_t|   \(|\wh \xi_t|  + \mathcal {W}_2  ( \bP_{ \xi_t }, \bP_{  \wt \xi_t  } )   \)  
  +     C_0 ^2  \(|\wh \xi_t|  + \mathcal {W}_2  ( \bP_{ \xi_t }, \bP_{  \wt \xi_t  } )    \) ^2  \) \kappa^{-1}\1_{\{  R_t=0  \}}      dt \bigg]   \nonumber\\
&\le&\EE \bigg[\int_{0}^T e^{-\beta t}   \bigg(   \frac { |\wh X_t| ^2} 4 + 4  C_0 ^2        \(|\wh \xi_t|  + \mathcal {W}_2  ( \bP_{ \xi_t }, \bP_{  \wt \xi_t  } )    \)  ^2 \bigg)   dt\bigg] \nonumber\\
&&+     \EE \bigg[\int_{0}^T  e^{-\beta t}   \Big(  \frac { |\wh X_t| ^2} 4 + 
 5 \kappa^{-1}   C_0 ^2         \(|\wh \xi_t|  + \mathcal {W}_2  ( \bP_{ \xi_t }, \bP_{  \wt \xi_t  } )   \)  ^2  \Big) dt \bigg]     \nonumber\\ 
 &=&2^{-1} \EE \bigg[\int_{0}^T e^{-\beta t}     |\wh X_t| ^2 dt\bigg] 
  +      C_0^2 (4+5\kappa^{-1})     \EE \int_{0}^T e^{-\beta t}  
       \(|\wh \xi_t| ^2 +   \mathcal {W}_2  ( \bP_{ \xi_t }, \bP_{  \wt \xi_t  } )^2      \)      dt\bigg] \nonumber\\
  &\le & 2^{-1} \EE \bigg[\int_{0}^T e^{-\beta t}  |\wh X_t| ^2  dt\bigg] 
  + 2 C_0^2 (4+5\kappa^{-1})     \EE \int_{0}^T e^{-\beta t}      |\wh \xi_t| ^2       dt\bigg] .
 \end{eqnarray} 
 Thus, 
\begin{eqnarray*}
  (2\beta -1)  \EE\bigg[   \int_{0}^T e^{-\beta t}     |\wh X_t| ^2     dt   \bigg]  
&\le &      4 C_0^2 (4+5\kappa^{-1})   
\EE \int_{0}^T e^{-\beta t}     |\wh \xi_t| ^2    dt   .
 \end{eqnarray*}
Choosing $\beta > 1 $ sufficiently large so that  
$$  
  \frac { 4 C_0^2 (4+5\kappa^{-1})   }  {2\beta -1} < 1/4. 
 $$
Then   $\Phi$ is a contraction map on the Banach space  
   on $({\mathcal M}[0, T], \| \cdot  \|_{{\mathcal M}_{\beta} [0,T] } ) $ with  
 $$
 \| \Phi (\xi ) -\Phi (\wt \xi) \|_{{\mathcal M}_{\beta} [0,T] }  \leq \tfrac12 \, \| \xi -\wt \xi \|_{{\mathcal M}_{\beta} [0,T] }
 \quad \hbox{for any } \xi, \wt \xi \in M[0,T]. 
 $$ 
 Hence $\Phi$ has a unique fixed point  $\bar X $ in ${\mathcal M}[0, T]$,
 which is the unique $L^2$ strong solution to  the  MF-SDE  \eqref{SDE}.  At last,   using Doob's $L^2$-maximal inequality, 
 we can deduce from the 
  Ito's formula applied to $|\wh X_t  |^2  $  and Hypothesis \ref{HP0-SDEs}   that 
  $\EE \Big[ \sup\limits_{0 \leq t \leq T}  |X_t|^2 \Big]< \infty$.
 \qed

   \section{ Existence and uniqueness of solution for  MF-BSDEs}\label{S:4}
   
   Let $n \geq 1$ be an integer. Denote by $L^2(\wt\sF_T; \R^n) $
    the space of square integrable  $\R^n$-valued $\wt \sF_T$-measurable random variables.
   To study stochastic maximum principle for  MF-SDEs \eqref{SDE},  
   we need to establish the existence and uniqueness 
of solutions to the following 
 MF-BSDE driven by sub-diffusions on $[0, T]$ for any $T> 0$:
\begin{equation}\label{e:BSDE}
 d Y_t   =h_1(\omega, t,Y_t,Z_t,\bP_{(Y_t,Z_t)})
 dt  + h_2 (\omega, t,Y_t,Z_t,\bP_{(Y_t,Z_t)})dL_t +Z_t d B_{L_t}  \quad \hbox{with }
Y_T  = \xi,
 \end{equation}
 where $\xi \in L^2(\wt\sF_T; \R^n) $, 
 $Y$ and $Z$ are $\R^n$-valued and $\R^{n\times d}$-valued $\wt \bF$-progressively measurable processes
and  $\bP _{( Y_t, Z_t)}  $  is the probability measure  (or law) induced  by   $( Y_t, Z_t)$. 
For each fixed $y\in \R^n$, $z \in \R^{n\times d}$, a probability measure  $\mu $ on $\R^n\times \R^{n\times d}$ and $i=1, 2$, 
  $(\omega, t)  \rightarrow h_i (\omega, t, y, z, \mu ) $  is an $\R^n$-valued   $\wt \bF$-progressively
  measurable process  defined on $\Omega$.
  Conditions for $h_i$ on variables $  y, z$ and $\mu$  will be imposed later in Hypothesis \ref{HP-BSDE}.  
For notational simplicity, we will typically drop $\omega$ from the expressions  of the above random processes or variables.

\begin{definition}\label{D:4.1} Let $T\in (0, \infty)$ and $R_t:=S_{L_t}-t$ be the overshoot process with $R_0=0$. 
\begin{enumerate} 
\item [\rm (i)] Denote by $\sM^2 [0,T]$ the space of 
 a pair $(Y, Z)$ of $\wt \bF$-progressively measurable processes
  on $[0, T]$ taking values in $\R^n \times \R^{n\times d}$ so that   
\begin{equation}\label{e:4.1a} 
Z_t = \1_{\{R_t=0\}} Z_t \  \hbox{ for } \ t\in [0, T] \quad \hbox{ and } \quad 
   \EE  \Big[    \int_{0}^T  |Y_t|^2   dt + \int_0^T |Z_t|^2 dL_t     \Big] <\infty .
\end{equation} 
For $\beta >0$, define a norm $\| (Y, Z)  \|_{ \sM^2_{\beta} [0,T]}$ on $\sM^2[0, T]$ by  
\begin{eqnarray}\label{e:4.2a}
\| (Y, Z)  \|_{ \sM^2_{\beta} [0,T]}  
:=  \( \EE\bigg[ \int_{0}^Te^{\be s}  |Y_s|^2ds + \int_0^T e^{\be s}|Z_s|^2 dL_s \bigg]\)^{1/2}  . 
\end{eqnarray}

\item [\rm (ii)]  A pair of process $ (Y, Z) $ is said to be 
 an $L^2$ adapted  solution of the BSDE \eqref{e:BSDE}  if  $ (Y, Z) \in \sM^2 [0,T]$ and
  for any $t\in [0, T]$,
\begin{equation}\label{e:BSDE2}
   Y_t   = \xi-\int_t^T h_1 (s,Y_s,Z_s,\bP_{(Y_s, Z_s)}) ds -\int_t^T   h_2  (s,Y_s,Z_s,\bP_{(Y_s, Z_s)})  dL_s
     -\int_t^T  Z_t d B_{L_s} .
  \end{equation}
\end{enumerate}
 \end{definition}

 \medskip

\begin{remark} \rm 
\begin{enumerate}
\item[(i)] Note that changing values of   the   process $Z$ over the random time intervals during which $L_t$ is flat (or equivalently,
when $R_t>0$) will not change the values of  the integrals by $dL_t$ and $dB_{L_t}$ in \eqref{e:BSDE2}, 
but it will change the value of the integral iby $dt$. So we need to specify the values of $Z_t$  
over the random time intervals during which $R_0>0$. It is natural to require $Z_t$  be zero when $R_t>0$
as in Definition \ref{D:4.1}(i)
for the space $\sM^2 [0, T]$. For a process $Z$ satisfying 
   $Z_s = \1_{\{R_s=0\}} Z_s$ $\bP$-a.s. for almost every $s\in [0, T]$,   
   we have by \eqref{e:2.3} that 
 \begin{equation}\label{e:4.9a}
 \E \int_0^T |Z_s|^2 ds = \kappa \E  \int_0^T  \kappa^{-1}  \1_{\{R_s=0\}} |Z_s|^2 ds 
 =  \kappa  \E  \int_0^T   |Z_s|^2 dL_s.
 \end{equation} 

\item[(ii)] Notice the different measurability requirement for $\sM [0, T]$ in Section \ref{S:3}
and $\sM^2[0, T]$. A process $X$ in $\sM [0, T]$ is $\bF'$-progressively measurable, 
while processes $(Y, Z)$ in $\sM^2 [0, T]$ are $\wt \bF$-progressively measurable. 
In addition, unlike the norm $\| \cdot \|_{ \sM_{\beta} [0,T]}  $ defined in \eqref{x-norm} on the space $\sM [0, T]$,
the norm  $ \| (Y, Z)  \|_{ \sM^2_{\beta} [0,T]} $ defined in \eqref{e:4.2a} on the space  $\sM^2 [0, T]$ carries
exponential weight $e^{\be s}$ instead of $e^{-\beta s}$. 
\qed
\end{enumerate} 
\end{remark}

\medskip

 \begin{hypothesis}\label{HP-BSDE} \rm 
 \begin{enumerate}
\item[\rm (i)]  $ 
  \EE \left[ \int_0^T  | h _1(s, {\bf 0},  \delta_{\bf 0} ) |^2 ds  +   \int_0^T | h_2(s, {\bf 0},  \delta_{\bf 0})|^2 dL_s  \right] <\infty,$
where ${\bf 0}$ denotes the origin in $\R^n\times \R^{n\times d}$  
	    and $\delta_{\bf 0}$ denotes the Dirac measure concentrated at {\bf 0}.
	      
 \item[\rm (ii)] For each $i=1, 2$, 
   $h_i (t, y, z, \mu)$ is uniformly Lipschitz continuous in $(  y, z,\mu) \in \R \times \R \times  \sP (\R^2)$    
   with Lipschitz constant $C_0>0$.
 That is, there is a $\Omega_0\subset \Omega$ with $\bP(\Omega_0)=1$ so that for every
 $\omega \in \Omega_0$, $t\in [0, T]$, $x_i, y_i, z_i\in \R$ and $\mu_i\in \mathcal {P}(\R^2)$  with $i=1,2$, 
\begin{equation}\label{e:4.5a} 
   |h_k (t,, y_1, z_1,\mu_1 ) -h_k(t, y_2, z_2,\mu_2)  | 
\leq  C_0 \, \big(  | y_1-y_2  | +	   | z_1-z_2  |   	+  \mathcal {W}_2(\mu_1, \mu_2) \big)
\quad \hbox{for } k=1, 2. 
  \end{equation}
   \end{enumerate}
  \end{hypothesis}

\begin{prop}\label{propBSDE}
Suppose that Hypothesis \ref{HP-BSDE}    holds and $(Y,Z)$ is an $L^2$ adapted  solution to   
 \eqref{e:BSDE}. Then   
\begin{eqnarray} \label{Est-BSDE}
   \EE \Big[ \sup\limits_{t\in[0,T]}  |  Y_t |^2    +      \int_{0}^{T} \vert Z_s \vert^{2} dL_s      \Big]  
    \leq    C\EE \left[   | \xi |^{2} + \int_{0}^{T}  | h_1  (s,   {\bf 0}, \delta_{\bf 0}  ) |^{2}ds+ 
   \int_{0}^{T} | h_2  (s,   {\bf 0}, \delta_{\bf 0}  ) |^{2} d  L_s  
   \right].  
     \end{eqnarray}
 \end{prop}

\pf   We divide the proof into several steps. 
Let $(Y,Z)$ be an $L^2$ adapted  solution to     \eqref{e:BSDE}.
Then
\begin{equation}\label{e:4.7a} 
Y_t= Y_0 + \int_0^t h_1 (s,Y_s,Z_s,\bP_{(Y_s, Z_s)}) ds +\int_0^t    h_2  (s,Y_s,Z_s,\bP_{(Y_s, Z_s)})  dL_s
     + \int_0^t  Z_t d B_{L_s} .
  \end{equation} 
     
  (i) In this step, we show 
 $      \E \Big[ \sup_{s\in [0, T]} |Y_s|^2 \Big] <\infty.$
    As 
   \begin{eqnarray*}
     \sup_{t\in [0, T]} |Y_t |  
     &\leq &  |Y_0| + \int_0^T  | h_1 (s,Y_s,Z_s,\bP_{(Y_s, Z_s)}) | ds +\int_0^T |    h_2  (s,Y_s,Z_s,\bP_{(Y_s, Z_s)}) | dL_s \\
&&     + \sup_{t\in [0, T]} \Big|  \int_0^t  Z_s d B_{L_s} \Big| \\
       &\leq &  |Y_0| + \int_0^T  | h_1 (s,{\bf 0},\delta_{\bf 0} ) | ds
       + \int_0^T  |  h_1 (s,Y_s,Z_s,\bP_{(Y_s, Z_s)}) -h_1 (s,{\bf 0},\delta_{\bf 0} ) | ds \\
       && + \int_0^T  | h_2 (s,{\bf 0},\delta_{\bf 0} ) | dL_s
       + \int_0^T  | h_2  (s,Y_s,Z_s,\bP_{(Y_s, Z_s)}) -h_2 (s,{\bf 0},\delta_{\bf 0} ) | dL_s \\
       && + \sup_{t\in [0, T]} \Big|  \int_0^t  Z_s d B_{L_s} \Big| , 
 \end{eqnarray*}
 we have by Hypothesis \ref{HP-BSDE}, \eqref{e:2.3}, \eqref{e:3.2}  and Doob's $L^2$-maximal inequality that 
 \begin{eqnarray} \label{e:4.8a}
     \E \Big[ \sup_{s\in [0, T]} |Y_s|^2 \Big] 
     &\leq & C_1 \E [|Y_0|^2] + C_1 \E \int_0^T   | h_1 (s,{\bf 0},\delta_{\bf 0} ) |^2  ds
     + C_1 \E \int_0^T   | h_2 (s,{\bf 0},\delta_{\bf 0} ) |^2  dL_s  \nonumber \\
     && + C_1 \E \int_0^T \left(|Y_s|^2 + |Z_s|^2 + \sW_2^2( \bP_{(Y_s, Z_s)}, \delta_{\bf 0} ) \right) ds
       + C_1 \E \int_0^T |Z_s|^2 dL_s   \nonumber \\
       &\leq & C_2 + C_3  \E \int_0^T \left(|Y_s|^2 + |Z_s|^2   \right) ds
       + C_1 \E \int_0^T |Z_s|^2 dL_s   \nonumber \\
       &\leq & C_2 + C_3  \E \int_0^T  |Y_s|^2   ds        + C_4 \E \int_0^T |Z_s|^2 dL_s ,
 \end{eqnarray}
 where the last inequality is due to \eqref{e:4.9a}. 
  This establishes the finiteness of   $\E \Big[ \sup_{s\in [0, T]} |Y_s|^2 \Big]$.

 (ii) We claim that 
 \begin{equation}\label{e:4.10a}
 M_t:= \int_0^t Y_s \cdot Z_s dB_{L_s}
 \  \hbox{ is a uniformly integrable $\wt \bF$- martingale.}
 \end{equation}
      Note that $M$ is a continuous local martingale with quadratic variation
  $\< M\>_t = \int_0^t | Y_s Z_s|^2 dL_s$. By  the  Burkholder-Davis-Gundy inequality, 
\begin{eqnarray*}\label{e:4.11a}
	\E \Big[ \sup_{t\in [0, T]}|M_t|\Big]
	&\leq&  C_5 \E \left[ \<M\>_T^{1/2}\right]   
	\leq C_5 \E \left[ \sup_{t\in [0, T] } |Y_t| \Big( \int_0^T |Z_s|^2 d L_s \Big)^{1/2} \right]  \\
	&\leq & \frac{C_5}{2} \E \Big[ \sup_{s\in [0, T]} |Y_s|^2  + \int_0^T |Z_s|^2 dL_s \Big] <\infty.   \nonumber
\end{eqnarray*}
	Hence $\{M_t; t\in [0, T]\}$ is a uniformly integrable martingale.

 (iii) We now proceed to show \eqref{Est-BSDE}. 
   By Ito's formula,  
 $$
 d |Y_s|^2 = 2 Y_s  \cdot dY_s + d \< Y, Y\>_s.
$$
    Integrating the above from $t$ to $T$,
we have by Hypothesis \ref{HP-BSDE} and \eqref{e:2.3},  
\begin{eqnarray} \label{e:4.13a} 
&&|Y_t|^2+ \int_t^T  | Z_s|^2d  L_s \nonumber  \\
& = & |\xi|^2 + \int_t^{T} 2  Y_s \cdot  h_1( s, , Y_s,Z_s, \bP_{(Y_s,Z_s)} )  ds 
+ \int_t^{T}  2Y_s \cdot h_2  ( s, , Y_s,Z_s, \bP_{(Y_s,Z_s)} )  dL_s + 2    \int_t^{T}   Y_s  \cdot Z_s d B_{L_s} \nonumber \\
 & \leq & |\xi|^2 +2 \int_t^T |Y_s|\, | h_1 ( s, {\bf 0}, \delta_{\bf 0} ) | ds + 2 \int_t^T |Y_s| \, | h_2 ( s, {\bf 0}, \delta_{\bf 0} )  | dL_s 
 \nonumber \\
&&  + (2+ \kappa^{-1})  C_0  \int_t^{T}    |Y_s| \(    |Y_s|  +   |Z_s|+
    \mathcal {W}_2  ( \bP_{( Y_t, Z_t)  }, \bP_{ \delta_{\bf 0}  } )   \) ds
   +   2 \int_t^{T}    Y_s  \cdot Z_sd B_{L_s}  .
\end{eqnarray}
Thus by Step (ii), \eqref{e:3.2},  Young's inequality and \eqref{e:4.9a},  
\begin{eqnarray*}
&& \EE\bigg[ |Y_t|^2+ \int_t^T  | Z_s|^2d  L_s\bigg] \nonumber\\
& \leq & \EE\bigg[ |\xi|^2  +  C_6 \int_t^{T} \(     |Y_s|^{2 } + |h_1( s, {\bf 0}, \delta_{\bf 0}) |^2 \) ds
+ C_6 \int_t^T  |h_2 ( s, {\bf 0}, \delta_{\bf 0} ) |^2 d L_s  
\nonumber \\
 &&\quad + \frac 1 2\int_t^T  | Z_s|^2d  L_s \bigg] 
\end{eqnarray*} 
for some $C_6>1$.
Hence we have
\begin{equation}\label{E-1} 
\EE\Big[  |Y_t|^2+  \frac12 \int_t^T  | Z_s|^2d  L_s\Big]
\leq C_6 \EE\Big[ |\xi|^2  +    \int_t^{T} \(     |Y_s|^{2 } + |h_1( s, {\bf 0}, \delta_{\bf 0}) |^2 \) ds
+   \int_t^T  |h_2 ( s, {\bf 0}, \delta_{\bf 0} ) |^2 d L_s  \Big].
\end{equation}
Dropping the second term on the left hand side and applying Gronwall's inequality  yields
\begin{eqnarray}\label{BSDE-Y-expec}
 \EE \Big[ |Y _t|^2 \Big]  \leq    C_7  \EE\bigg[ |\xi|^2 +  \int_{0}^{T}  
  | h_1( s, {\bf 0}, \delta_{\bf 0} ) |^2 ds+ \int_0^T | h_2 ( s, {\bf 0}, \delta_{\bf 0} ) |^2 d   L_s   \bigg].
 \end{eqnarray} 
This together by taking   $t=0$  in  \eqref{E-1} gives 
 \begin{eqnarray}\label{BSDE-YZ-expec}
  \EE \bigg[  \int_0^T   |Z_s|^2d  L_s\bigg] 
 \leq     C_8  \EE\bigg[ |\xi|^2 +  \int_{0}^{T}   |h_1 ( s, {\bf 0}, \delta_{\bf 0} )|^2 ds
 + \int_0^T | h_2 ( s, {\bf 0}, \delta_{\bf 0})  |^2 d   L_s   \bigg] .
 \end{eqnarray}
Thus  by \eqref{e:4.13a}, \eqref{e:2.3},  \eqref{e:3.2}, \eqref{e:4.9a}, Young's inequality and Doob's $L^2$-maximal inequality,   
\begin{eqnarray*} 
&&\EE \left[ \sup\limits_{0\le t \le T} |Y_t|^2 \right] \\
&\leq & \EE \Big[ |\xi|^2 + \int_0^T |Y_s|\, | h_1 ( s, {\bf 0}, \delta_{\bf 0} ) | ds 
+ \int_0^T |Y_s| \, | h_2 ( s, {\bf 0}, \delta_{\bf 0} )  | dL_s 
 \nonumber \\
&&  + (2+ \kappa^{-1})  C_0  \int_0^{T}     |Y_s| \(    |Y_s|  +   |Z_s|  + \sqrt{ \E [ |Y_s|^2+|Z_s|^2 ] }  \) ds
   +   2 \sup_{t\in [0, T]} |M_T -M_t| \Big] \\
 & \leq & \EE\bigg[ |\xi|^2  +  C_9 \int_{0}^{T} \(     |Y_s|^{2 } + |h_1 ( s, {\bf 0}, \delta_{\bf 0} )|^2 \) ds
+ C_9 \int_0^T  |h_2 ( s, {\bf 0}, \delta_{\bf 0} ) |^2 d L_s 
\nonumber \\
 &&+ C_9\int_0^T  | Z_s|^2d  L_s + C_9 \bigg(  \sup\limits_{0\le t \le T} 
 |Y_t|^2  \int_0^T  | Z_s|^2d  L_s   \bigg)^{1 / 2} \bigg] \nonumber\\
& \leq & \EE\bigg[ |\xi|^2  +  C_9 \int_{0}^{T} \(     |Y_s|^{2 } + |h_1( s, {\bf 0}, \delta_{\bf 0} )|^2 \) ds
+ C_9 \int_0^T  |h_2 ( s, {\bf 0}, \delta_{\bf 0} ) |^2 d L_s \bigg] 
\nonumber \\
 &&+ C_{10} \EE \int_0^T  | Z_s|^2d  L_s +  
\frac 1 2 \EE \left[ \sup\limits_{0\le t \le T} |Y_t|^2 \right].  
 \end{eqnarray*}
 It now follows from  Step (i), \eqref{BSDE-Y-expec}  and  \eqref{BSDE-YZ-expec} that
 \begin{eqnarray*} 
\EE \left[  \sup\limits_{0\le t \le T} |Y_t|^2 \right] 
&\leq &  2  \EE\bigg[ |\xi|^2  +  C_9 \int_{0}^{T} \(     |Y_s|^{2 } + |h_1 ( s, {\bf 0}, \delta_{\bf 0}  )|^2 \) ds
+ C_9 \int_0^T  |h_2  ( s, {\bf 0}, \delta_{\bf 0} )|^2 d L_s   \\
&& \qquad   + \,  C_{10} \int_0^T  | Z_s|^2d  L_s \bigg]  \\
&\le&C_{11}  \EE\bigg[ |\xi|^2  +    \int_{0}^{T}  |h_1 ( s,  {\bf 0}, \delta_{\bf 0}  )|^2  ds+  
 \int_0^T  |h_2  ( s,  {\bf 0}, \delta_{\bf 0}  ) |^2 d L_s \bigg] .
 \end{eqnarray*}
 This together with \eqref{BSDE-YZ-expec} establishes the desired estimate \eqref{Est-BSDE}. 
  \qed

\medskip

\begin{theorem}\label{T:BSDE}
Suppose Hypothesis \ref{HP-BSDE} holds. For any given $\xi \in  L^2 (\wt\sF_T;\R^n)$,
  the BSDE  \eqref{e:BSDE2}  admits unique $L^2$ adapted solution $(Y , Z ) \in \sM^2 [0,T]$
    in the sense that if $(\wt Y, \wt Z ) \in \sM^2 [0,T]$ is another solution of \eqref{e:BSDE2},
  then $\wt Y_t=Y_t$ for all $t\in [0, T]$ with probability one and $\E \int_0^T |Z_s -\wt Z_s|^2 dL_ s    =0$.
\end{theorem}

\pf      Let  $\xi \in L^2 (\wt\sF_T)$.    Given $(y_t, z_t)\in  \sM^2 [0,T]$,
 consider the following BSDE:
\begin{eqnarray}\label{e:3.19}
 d Y_t   =&h_1  (t, y_t, z_t,  \bP_{(y_t,z_t)} ) dt + h_2 (t, y_t, z_t,  \bP_{(y_t,z_t)} ) d L_t   +Z_t d B_t   \quad \hbox{with }  \ 
Y_T  =   \xi   .
\end{eqnarray}
 Define 
 $$
 \eta = \xi- \int_0^T h_1(s, y_s, z_s,    \bP_{(y_s, z_s)} ) ds-\int_0^T  h_2 (s, y_s, z_s,    \bP_{(y_s, z_s)} ) dL_s ,
 $$
 which is $\wt \sF_T$-measurable. 
By Hypothesis  \ref{HP-BSDE},  \eqref{e:2.3}, \eqref{e:3.2} and \eqref{e:4.9a}, 
\begin{eqnarray*}
\E [\eta^2]
 &\leq& 3\E [ |\xi|^2] + 3T \E \int_0^T |h_1 (t, y_t, z_t,    \bP_{(y_t, z_t)} )  |^2 dt 
 +  \frac {3T} {\ka}\E  \( \int_0^T |h_2(t, y_t, z_t, \bP_{(y_t,z_t)}  ) |^2 dL_t \)    \\
&\leq & 3\E [ |\xi|^2] + 3T  C_1  \E \int_0^T \(|h_1(t, {\bf 0} , \delta_{\bf 0})|^2+|y_t| ^2 + |z_t|^2\) dt \\
&& +  \frac {3TC_1 } {\ka }    \Big(    \kappa^{-1} \E  \int_0^T \( |h_2  (t, {\bf 0} , \delta_{\bf 0}) |^2 + |y_t|^2\) dt 
+  \EE \int_0^T   |z_t|^2  dL_t  \Big) \\
&<& \infty.
\end{eqnarray*}
 Thus by  Theorem \ref{T:2.2}, there exists an $ \R^{n\times d}$-valued   
 $\wt \bF$--progressively measurable process $Z$
  with   $\E \int_0^T  |Z_s|^2 d L_s  <\infty$ so that
 $\eta = \E [\eta|  \wt \sF_0 ] + \int_0^T Z_s d B_{L_s    }$.

Define 
$$ Y_t= \E [ \eta | \wt \sF_0] 
+ \int_0^t h_1 (s, y_s, z_s ,  \bP_{(y_s,z_s)}) ds + \int_0^t  h_2 (s, y_s, z_s ,  \bP_{(y_s,z_s)}) dL_ s  +   \int_0^t  Z_s dB_{L_ s   }.
$$
Then by the same calculation as above, $\E \int_0^T |Y_t|^2 dt <\infty$. Moreover, for $t\in [0, T]$, 
\begin{eqnarray*}
 Y_t &=& \eta -  \int_0^T Z_s d B_{L_s    }  + 
 \int_0^t h_1 (s, y_s, z_s ,  \bP_{(y_s,z_s)}) ds + \int_0^t  h_2 (s, y_s, z_s ,  \bP_{(y_s,z_s)}) dL_ s  +   \int_0^t  Z_s dB_{L_ s   } \\
 &=&  \xi  -  \int_t^T h_1 (s, y_s, z_s ,  \bP_{(y_s,z_s)}) ds - \int_t^T  h_2((s, y_s, z_s ,  \bP_{(y_s,z_s)}) dL_ s 
 -  \int_t^T  Z_s dB_{L_ s   } .
 \end{eqnarray*}
 Thus  $(Y_t, Z_t) \in  \sM^2 [0,T]$ solves BSDE  \eqref{e:3.19}. 
Suppose that $(Y'_t, Z'_t) \in  \sM^2 [0,T]$ is another solution of BSDE  \eqref{e:3.19}.
Then $d(Y_t-Y'_t)= (Z_t-Z'_t) dB_{L_t   }$ with $Y_T-Y'_T=0$. 
  It follows  that $\E \int_0^T(Z_s-Z_s')^2 d L_ s =0$ and, 
  consequently,  $Y_t=Y_t'$ for all $t\in [0, T]$ $\bP$-a.s. as both $Y_t$ and $Y_t'$ are continuous processes. 
  This shows that BSDE  \eqref{e:3.19}  has a unique $L^2$ adapted  solution.
   
\medskip
 
 The above defines a map    $\Phi:  \sM^2  [0,T]\rightarrow  \sM^2   [0,T] $ by sending $(y,z)$ to the unique $L^2$ adapted  
 solution $ (Y,Z)$  of \eqref{e:3.19}.   We next show that it is a contractive map with respect to the Banach norm $\| \cdot \|_{ \mathcal  M_{\beta} [0,T]}$  for   $\beta >1$ sufficiently large.  
 For $(y,z)$ and $(\wt y, \wt z)\in  \sM^2 _\be [0,T]$, let $(Y, Z)= \Phi (y, z)$ and $(\wt Y, \wt Z) = \Phi ( \wt y, \wt z)$.
 For notational simplification, let
 $$
 \wh Y_t:= Y_t - \wt Y_t, \quad \wh Z_t:= Z_t - \wt Z_t, \quad   \wh y_t:= y_t - \wt y_t, \quad \wh z_t:=z_t - \wt z_t,
 $$
 and
 $$
 \wh h_1 (t):=  h_1(t,   y_t,   z_t , \bP_{ (y_t,z_t)}  ) - h_1(t, \wt y_t, \wt z_t, \bP_{ (\wt y_t, \wt z_t)} )  , 
  \quad  \wh h_2 (t):=  h_2(t,   y_t,   z_t , \bP_{ (y_t,z_t)}  )  - h_2(t, \wt y_t, \wt z_t, \bP_{ (\wt y_t, \wt z_t)}   )  .
 $$
   By \eqref{e:3.19},
 $$
 \wh Y_t = - \int_t^T \wh h_1 (s) ds - \int_t^T \wh h_2 (s) dL_s    - \int_t^T \wh Z_s dB_{L_ s    } .
$$
 Let $\beta>1$, whose value will be taken to be sufficiently large later. 
 By Ito's formula, 
 $$ 
 d( e^{\be s} |\wh Y_s  |^2 ) = \beta  e^{\be s}|\wh Y_s|^2 ds 
 +  e^{\be s}  \( 2 \wh Y_s d\wh Y_s + d \< \wh Y, \wh Y\>_s\).
 $$  
  Since $\wh Y_T =0$,  Integrating the above from $t$ to $T$  yields
  \begin{eqnarray*}
 && e^{\beta t}  |\wh Y_t|^2  + \int_t^T  e^{\be s}  |\wh Z_s|^2 dL_ s    \\
 &=&  -  \int_t^T e^{\be s} \left(  \beta |\wh Y_s|^2 + 2\wh Y_s \wh h_1 (s)  \right) ds 
 -  2 \int_t^T   e^{\be s}  \wh Y_s  \wh h_2 (s)  dL_  s
 - 2 \int_t^T   e^{\be s} \wh Y_s \wh Z_s  dB_{L_ s    } .
   \end{eqnarray*}
  Let 
$M_t:=2  \int_t^T   e^{\be s} \wh Y_s \wh Z_s  dB_{L_ s    } $.
By the same argument as that for \eqref{e:4.10a}, we see  $\{M_t; t\in [0, T] \}$ is a uniformly integrable martingale.
  From the last display, we have by Hypothesis  \ref{HP-BSDE}, \eqref{e:2.3} and \eqref{e:3.2}, 
  \begin{eqnarray}\label{esti-norm1}
&& e^{\beta t}  |\wh Y_t|^2  + \int_t^T  e^{\be s}  |\wh Z_s|^2 dL_ s       \nonumber  \\
  &\le&\int^T_t\big(-\beta  | \wh Y_s |^2+C_2 |\wh{Y}_s | \( |\wh{y}_s| + |\wh{z}_s|
    + \sW_2(\bP_{(y_t, z_t)}, \bP_{(\wt y_t,\wt z_t)} ) \)\big)e^{\be s}ds \nonumber\\
&&+\int^T_t C_2 |\wh{Y}_s| \( |\wh{y}_s|+|\wh{z}_s|+ \sW_2(\bP_{(y_t, z_t)}, \bP_{(\wt y_t,\wt z_t)} )  \)e^{\be s} dL_  s
 +  ( M_t -M_T)  \nonumber\\
 &\le&   \int^T_t\(-\beta | \wh Y_s |^2+ C_3 |\wh{Y}_s|   \( |\wh{y}_s| +  |\wh{z}_s|  
 + \sqrt{ \E \left [  |\wh{y}_s|^2   +  |\wh{z}_s|^2 \right] } \)  \) e^{\be s} ds  + (M_t-M_T)   \nonumber \\
&\leq &    \int^T_t \(-\beta  | \wh Y_s |^2+  \beta | \wh Y_s |^2 +
\frac{ C_3^2}{4\beta}      \( |\wh{y}_s| +  |\wh{z}_s |    + \sqrt{ \E \left [  |\wh{y}_s|^2   +  |\wh{z}_s|^2 \right] } \)^2  \)   e^{\be s} ds   
   \nonumber  \\
&& + (M_t-M_T)   \nonumber \\
&\leq &  \frac{3 C_3^2}{4\beta}     \int^T_t   \( |\wh{y}_s|^2 +  |\wh{z}_s|^2  +  \E \left[  |\wh{y}_s|^2    +   |\wh{z}_s|^2  \right] \)
   e^{\be s} ds   + (M_t-M_T) 
 \end{eqnarray}
 Taking expectation on both sides and using the property $\wt z_s = \1_{\{\R_t=0\}} \wt z_s$ and \eqref{e:2.3},
 we get 
   \begin{equation}\label{e:4.19a}
 \E \left[  e^{\beta t}  |\wh Y_t|^2  + \int_t^T  e^{\be s}  |\wh Z_s|^2 dL_ s   \right]    
  \leq \frac{3 C_3^2 (\kappa+1) }{2 \beta} \( \E  \int^T_t     |\wh{y}_s|^2 e^{\be s} ds +
  \E \int_t^T  |\wh{z}_s|^2   e^{\be s} dL_s\) . 
 \end{equation}
  Dropping the first term on left hand side of \eqref{e:4.19a} and then taking $t=0$ gives  
  \begin{equation}\label{e:4.20a}
 \E  \int_0^T  e^{\be s}  |\wh Z_s|^2 dL_ s 
 \leq   \frac{3C_3^2 (\kappa+1)}{2 \beta} \| (\wh y, \wh z)\|_{\sM^2_\beta [0, T]}^2.
  \end{equation}
  On the other hand, by \eqref{esti-norm1} , Doob's $L^2$-maximum inequality and \eqref{e:2.3}, 
 \begin{eqnarray}\label{esti-norm3}
 &&   \EE \left[ \sup_{t\in [0, T]} \left( e^{\beta t} |\wh{Y}|^2_t  \right) \right]  \nonumber\\
&\leq &   \frac{3 C_3^2}{4\beta}  \E   \int^T_0   \( |\wh{y}_s|^2 +  |\wh{z}_s|^2  +  \E [  |\wh{y}_s|^2    +   |\wh{z}_s|^2 ] \)
   e^{\be s} ds      
          + \E \left[ \sup_{t\in [0, T]} | M_t-M_T|  \right]  \nonumber\\
&\leq &   \frac{3C_3}{ 2 \beta}   \E   \int^T_0   \( |\wh{y}_s|^2  +  |\wh{z}_s|^2 \1_{\{R_t=0\}}  \)  e^{\be s} ds      
          + 2 \E \left[ \sup_{t\in [0, T]} | M_t|  \right]  \nonumber\\
 &\leq &  \frac{3 C_3^2 (\kappa+1) }{2 \beta} \| (\wh y, \wh z)\|_{\sM^2_\beta [0, T]}^2 
 +  8 \EE \left[  \left(\int^T_0e^{2\be s} |\wh Y_s |^2  \, |\wh Z_s |^2 dL_s    \right)^{1/2} \right] 
 \nonumber\\
   &\leq &  \frac{3 C_3^2 (\kappa+1) }{2\beta} \| (\wh y, \wh z)\|_{\sM^2_\beta [0, T]}^2 
 +  8  \EE \left[  \left(\sup_{t\in [0, T]} \left(e^{ \be t}|\wh{Y}_t|^2 \right) \int^T_0e^{\be s} |\wh Z_s |^2d  L_s \right)^{1/2} \right] \nonumber\\
 &\leq &  \frac{3 C_3^2 (\kappa+1) }{2\beta} \| (\wh y, \wh z)\|_{\sM^2_\beta [0, T]}^2 
 +   32   \E \int_0^T e^{\be s} |\wh Z_s |^2d  L_s
 +\frac12\EE \left[\sup_{t\in [0, T]}\left(e^{\beta t} |\wh{Y}|^2_t\right) \right], \nonumber\\
 &\leq &  \frac{99 C_3^2 (\kappa+1) }{2 \beta} \| (\wh y, \wh z)\|_{\sM^2_\beta [0, T]}^2  
  +\frac12\EE \left[\sup_{t\in [0, T]}\left(e^{\beta t} |\wh{Y}|^2_t\right) \right].
\end{eqnarray}
where the last inequality is due to  \eqref{e:4.20a}.
Since $\E  \left[\sup_{t\in [0, T]}\left(e^{\beta t} |\wh{Y}|^2_t\right) \right] <\infty$ in view of Proposition
\ref{propBSDE}, we conclude 
\begin{equation}\label{esti-norm4}
\EE \left[ \sup_{t\in [0, T]}\left(e^{\beta t} |\wh{Y}|^2_t\right) \right]
\le    \frac{99 C_3^2 (\kappa+1) }{\beta} \| (\wh y, \wh z)\|_{\sM^2_\beta [0, T]}^2  .
\end{equation}
   Consequently, 
\begin{equation}\label{e:4.11}
\EE \left[ \int_0^T e^{\beta t} |\wh{Y}|^2_t dt  \right]
\leq T \EE \left[ \sup_{t\in [0, T]}\left(e^{\beta t} |\wh{Y}|^2_t\right) \right]
\le  \frac{99 C_3^2 (\kappa+1) T }{\beta} \| (\wh y, \wh z)\|_{\sM^2_\beta [0, T]}^2  .
\end{equation}
   Combining this with \eqref{e:4.20a}, we have
\[\|(\wh{Y},\wh{Z})\|^2_{\cM^2_\be [0, T]} \leq
   \frac{101 C_3^2 (\kappa+1) (T +1)}{\beta} \| (\wh y, \wh z)\|_{\sM^2_\beta [0, T]}^2  .
   \]
   Taking $\be >1$ sufficiently large so that $ \frac{101 C_3^2 (\kappa+1) (T +1)}{\beta} \leq 1/4$, 
   we get 
   $$
   \|(\wh{Y},\wh{Z})\|_{\cM^2_\be [0, T]} \leq
   \tfrac12  \| (\wh y, \wh z)\|_{\sM^2_\beta [0, T] }  .
   $$
This shows that   $\Phi$ is a contraction map 
   on the Banach space $( \sM^2 [0, T],  \| \cdot \|_{\sM^2 _{\be} [0, T] } ) $. 
   Hence $\Phi$ has a unique fixed point  $(\bar Y, \bar Z)$ in $\sM^2 [0, T]$,
 which is the unique $L^2$ adapted  solution to  the  MF-BSDE  \eqref{e:BSDE2}.   
   \qed

 \section{Control problem for  MF-SDEs}\label{S:5}

 We recall the following definition from \cite{ZhangChen2024SICON1, ZC4}. 

\begin{definition}\label{D:2.1}  \rm
 Let $U$ be a  non-empty convex    subset (i.e., an interval) of $\R^n$. For each $T>0$ and  $a\geq 0$, denote by $\mathcal{U}' [0, T]$ the set of  
    $\bF'$-progressively measurable processes  $\{u (t, \omega );  t\in [0, T]\}$ defined on $   [0, T] \times \Omega $  taking values in $U$ so that 
   $\E \int_0^T |u_t|^2 dt <\infty$, 
    where $\bF'$ is the natural augmented filtration generated by the sub-diffusion $B_{L_t}$.

We call $u\in \mathcal{U}'[0, T] $  an admissible control.
   Note that  the space $\mathcal{U}'[0, T]$ depends on the convex set $U$ but for notational convenience we do
  not include $U$ in its notation. In the following we call $U$ a control domain. Observe that $\mathcal{U}' [0, T]$ is convex as $U$ is convex. 
 \end{definition}

\smallskip

\medskip

Given $u\in \mathcal{U}'  [0, T]$  and $x_0\in \R^n $, the state process $X=X^u$ is described by  the following  
 mean-field SDE on $\R^n$ driven 
 by the anomalous sub-diffusion $B_{L_t}$ for $t\in [0, T]$:    
  \begin{equation}\label{e:SDE}
\left\{\begin{aligned}
 dX_t =&\ \EE ^\prime \left[ b(t,   X_  t , X_t^\prime, u_t) \right] dt  
 +\ \EE ^\prime \left[ \delta (t,   X_  t , X_t^\prime, u_t) \right] dL_t
+\EE^ \prime \left[ \sigma(  t,   X_  t , X_t^\prime, u_t   ) \right] dB_{L_t}  ,   \\
 X_0=&\ x_0 , 
\end{aligned}
\right.
\end{equation}
where $\varphi: \Omega \times [0,T]\times\RR^n \times\RR^n \times U  \rightarrow  \RR^n  $,
with $\varphi = b, \delta , \sigma$,  is such that 
 for each fixed $(x, y, u)\in \R^n \times \R^n \times U$, $(\omega, t) \mapsto \varphi  (\omega, t, x, y, u)$
  is $\bF'$-progressive measurable process. Here 
  $X'$ is an independent copy of $X$ and $\E'$  is the expectation taken under the law of $X'$, that is, 
 $\EE ^\prime \left[ \varphi (t,   X_  t , X_t^\prime, u_t) \right]  = \int_{\R^n} \varphi (t, X_t, y, u_t) \bP_{X_t} (dy)$.

   \medskip

 \begin{theorem}\label{T:5.3}
 Suppose that $b(t, x, y, u)$, $\delta (t, x, y, u)$ and $\sigma (t, x, y, u)$ are Borel measurable  functions
 on $[0, \infty) \times \R^n \times \R^n \times U$  and  there is a constant
 $C_0 \geq 1 $ such that  for  any $t\in [0, T] $ and $x_i,  y_i , u \in \R^n$ with $i=1, 2$, 
 \begin{equation} \label{e:5.2}
 | \varphi (t, x_1, y_1, u) - \varphi (t, x_2, y_2, u)|     \leq C_0 \( |x_1 -x_2 | + |y_1-y_2|\), 
   \end{equation}
   and
 \begin{equation}\label{e:5.3}
 | \varphi (t, x_1, y_1, u)  |     \leq C_0 \( 1 + |u|\)
   \end{equation}
  with  $\varphi = b, \delta, \sigma$. 
 Then for every $u\in \cU' [0, T]$ and $x\in \R^n$, the  MF-SDE \eqref{e:SDE} has a unique $L^2$ strong  solution $X$.
   \end{theorem}

 \pf Let  $u\in \cU' [0, T]$. For $t\geq 0$, $x\in \R^n$ and $\mu \in \sP (\R^n)$, 
  define 
  $$
  \wt \varphi (\omega, t, x, \mu):= \int_{\R^n} \varphi (t, x, y,  u_t (\omega)) \mu (dy) 
   $$
 for $\varphi = b, \delta , \sigma$. 
 Clearly, for  each fixed $x, x \in \R^n$ and    $\mu_1 , \mu_2\in \sP(\R^n)$, 
 $(\omega, t) \rightarrow \varphi (\omega, t, x,  \mu) $    is
   $\bF'$-progressively measurable with  $\varphi = b, \delta,  \sigma $.
  Moreover,  by \eqref{e:2.3}, \eqref{e:5.2} and \eqref{e:5.3},  
  \begin{eqnarray} \label{e:5.4} 
 && \E \left[ \int_0^T \wt b (s, 0, \delta_{0})^2 ds  + \int_0^T \(   
    \wt \delta (s, 0, \delta_{0})^2 + \wt \si  (s, 0,  \delta_{0})^2    \) dL_s  \right]  \nonumber \\
    &\leq & C_0^2 \(1+2\kappa^{-1} \) \E \int_0^T (1 + |u(s)|^2) ds <\infty
  \end{eqnarray}
  and
   \begin{eqnarray*}\label{e:5.5}
  && | \wt \varphi (\omega, t, x_1, \mu_1) -\wt \varphi (\omega, t, x_2, \mu_2)|   \nonumber\\
  &=& \inf_{\pi \in \sG^L (\mu_1, \mu_2)} \Big| \int_{\R^n\times \R^n} \( b(t, x_1, y_1, u(t)) - b(t, x_2, y_2, u(t)) \)  \pi (dy_1, dy_2) \Big| \nonumber\\
  &\leq &\inf_{\pi \in \sG^L (\mu_1, \mu_2)}   \int_{\R^n\times \R^n} |  b(t, x_1, y_1, u(t)) - b(t, x_2, y_2, u(t)) |  \pi (dy_1, dy_2) \nonumber\\
  &\leq &\inf_{\pi \in \sG^L (\mu_1, \mu_2)}   \int_{\R^n\times \R^n} C_0 \( |x_1-x_2|+|y_1 -y_2|  \)  \pi (dy_1, dy_2) \nonumber\\
  &=& C_0 \( |x_1-x_2|+ \sW_1 (\mu_1, \mu_2)  \) \nonumber \\
  &\leq & C_0 \( |x_1-x_2 |+ \sW_2 (\mu_1, \mu_2)  \).
     \end{eqnarray*} 
     Thus $(\wt b, \wt \delta, \wt \sigma)$ satisfy Hypothesis \ref{HP0-SDEs}.

Note  \eqref{e:SDE} can be rewritten as the following form:
  \begin{equation}\label{SDE2}
\left\{\begin{aligned}
dX_t =&   \wt b(t,   X_  t , \bP_{X_t} )  dt  + \wt \delta (t,   X_  t , \bP_{X_t} )  dL_t
+ \wt \sigma(  t,   X_ t , \bP_{X_t}  )   dB_{L_t}  \quad \hbox{for } t\in[0,T],    \\
 X_0=&\ x_0 .
 \end{aligned}
\right.
 \end{equation}
 By   Theorem \ref{exis-unique-SDEs}, \eqref{SDE2} has a unique $L^2$ strong solution, 
 so does \eqref{e:SDE}. \qed

 \medskip
Now suppose $b, \delta$ and $\sigma$ are deterministic. 
 We consider the following cost functional  for control $u\in \cU' [0, T]$:
   \begin{eqnarray}\label{e:5.8}
J(u)&:=& \EE   \bigg[     \int_{0}^{T}\EE'  \left[ f\left  (t, X_t, X_t ^{\pr} ,    u_t   \right) \right]  dt
 + \int_{0}^{T}\EE'  \left[ g \left  (t, X_t, X_t ^{\pr} ,    u_t   \right) \right]  dL_t
+\EE ^{\pr}h(X_T, X_T  ^{\prime}    )  \bigg],
\end{eqnarray}
 where $   f , g   :\;[0,T]\times\R^n  \times \R^n   \times U   \to \R  $ and 
 $   h   :\; \R^n  \times \R^n      \to \R  $.    Note that $J(u)$ depends on the initial value $x_0\in \R^n$ of the state processes,
 which is fixed.

\medskip

An admissible control $u^\ast \in \cU'[0, T]$ is said to be an optimal control for the cost functional
$J$ if 
\begin{equation}\label{control}
 J(u^\ast )  =\inf_{u \in \mathcal{U}' [0, T]} J( u ) . 
\end{equation}

In next section, we study the stochastic maximum principle for the above control problem. 
 
\section{ Stochastic Maximum Principle}\label{S:6} 

 To study the stochastic maximum principle  for \eqref{control}, we need to introduce an adjoint equation,
which is a  MF-BSDE that is a closely related to but is slightly different from  \eqref{e:BSDE}. 
For this, 
 let $    (\bar  {\Omega}, \bar {\mathcal{F} },  \bar{ \mathbb P } ) 
 =   ( \Omega\times \Omega  ,  \mathcal{F} \times \mathcal{F},    \mathbb P \times  \mathbb P  )    $
 be the product   probability space  of $ (\Omega, \mathcal{F}, \mathbb P)$ with itself. 
 We equip this product space with filtration $\bar \bF:=\{ \bar \sF_t:= \wt \sF_t \otimes \wt \sF_t, t\in [0, T]\}$. 
 For a random variable $\xi$ that is originally defined on $\Omega$, it can be naturally as a random variable,
 still denoted by $\xi$, 
 on $\bar \Omega$ by $\xi (\omega, \omega'):= \xi (\omega)$. 
 It also defines an independent copy $\xi'$ on $\bar \Omega$ by $\xi' (\omega, \omega'):= \xi (\omega')$.
 Similar remark applies to random processes as well. 
   For $\theta \in L^1 (  \Omega\times \Omega  ,  \mathcal{F} \times \mathcal{F},    \mathbb P \times  \mathbb P  )$, 
 we denote
 \[ 
 \EE^{\pr} [  \theta(\omega, \cdot)  ] :=\int_{\Omega} \theta(  \omega , \omega^{\pr} ){ \mathbb P }(d\omega') .
 \]
 Clearly, 
 \[
 \bar\E [  \theta ] :=\int_{\bar\Omega}\theta d\bar { \mathbb P } 
 =  \int_{\Omega}  \EE^{\pr} [  \theta(\omega, \cdot) ]    \mathbb P  (d\omega )  =\EE [\EE '  [\theta]]    .   
  \] 

\medskip

Since     the control domain  
  $U\subset \RR $ is convex,  so is $\cU' [0, T]$.  
  In this section,  we establish stochastic maximum principle  for \eqref{control}  using   a  convex variational method.
  Let $0<T<\infty$. 
 Throughout this section, we assume the following Hypothesis  holds. 
 
  \medskip
 
\begin{hypothesis}\label{HP1}  \rm 
 
  The functions $b(t, x,  y,     u)$, $\delta (t, x, y,  u)$,   $\si (t, x, y,  u),$  $f(t, x, y,   u)$  and    $h (  x, y)$
    are deterministic  continuously  
differentiable in $(x, y, u)$ and in $(x, y )$, respectively,
  with bounded and Lipschitz continuous first order partial derivatives. 
 Denote by $C_0>0$ the least bound of all these first order partial derivatives.

\end{hypothesis}

 For $u\in \cU'[0, T] $, denote by 
 $  X^u  $ the unique solution of 
\eqref{e:SDE}  in $\mathcal M [0, T]$.  
In the following, for notational simplicity, 
denote 
\begin{equation}\label{e:Theta}
\Theta^{u}_t := (X^u_t,  (X^u_t)^\prime,  u_t) 
\quad \hbox{ and } \quad 
\widehat\Theta^{u}_t := (  (X^u_t)^\prime,  X^u_t,   (u_t)') .  
\end{equation}

\bigskip
 
For $\bar u, u\in  \cU ' [0, T] $, set  $v:=u-\bar u$. Then  
$$
\bar u+\va v = (1-\eps) \bar u + \eps u \in  \cU ' [0, T]  \qquad \hbox{for every } \eps \in [0, 1].
$$

\begin{lemma}\label{xpi-x} 
Suppose Hypothesis  \ref{HP1}  holds. 
    Then for every $\bar u,v' \in  \cU ' [0, T] $,  
there exits a constant $C$ such that    with  $v= v'-\bar u$,
  \[    \EE \left[  \sup\limits_{0\leq t \leq T}     |X^{ \bar u+\va v }_t -X^{\bar u}_t |^2 \right] \leq C \va ^2   \E \int_0^T |v_s|^2 ds
  \quad \hbox{ for every }   \eps \in (0, 1).
   \]
  \end{lemma}

\pf  Note that Hypothesis \ref{HP0-SDEs} holds under the assumption of Hypothesis  \ref{HP1}. Thus 
by Theorem \ref{exis-unique-SDEs}, $ \EE \Big[  \sup\limits_{0\leq t \leq T}     |X^{ \bar u+\va v }_t -X^{\bar u}_t |^2 \Big] 
<\infty.$  
 By   Burkholder-Davis-Gundy's   inequality, we have
 \begin{eqnarray*}
 && \EE \left[ \sup\limits_{0\leq s \leq t}   |X^{ \bar u+\va v } _s   -X^ {\bar u} _s  |^2  \right] \nonumber\\
& \leq &C_1\EE  \bigg[  \(\int_0^t \EE^\prime  |b(s,  \Theta^{\bar u + \eps v}_s) 
-b(s,   \Theta^{\bar u}_s )| ds\)^2 
 + \(\int_0^t \EE^\prime  |\delta(s,  \Theta^{\bar u + \eps v}_s) 
 -\delta(s,   \Theta^{\bar u}_s )| dL_s\)^2
\nonumber\\
&&+  \int_0^t  \EE^\prime   |\si(s,  \Theta^{\bar u + \eps v}_s) 
-\si(s,   \Theta^{\bar u}_s )|  ^2 d\langle B_{L_s}\rangle   \bigg] \nonumber\\
  & \leq &  C_2  \EE \left[ \(  \int_0^t   \left( \EE^\prime |   (X^{\bar  u+\va v } _s)^\prime  -(X^ {\bar  u } _s)^\prime     |+     |   X^{\bar  u+\va v } _s  -X^ {\bar  u } _s     |+ \va  | v_s|   \right)    ds \)^2 \right] \nonumber\\
 && +  C_2   \EE  \left[ \int_0^t \( \EE^\prime |   (X^{\bar  u+\va v } _s)^\prime  -(X^ {\bar  u } _s)^\prime     |+  |   X^{ \bar  u+\va v } _s  -X^ {\bar   u }_s      |+ \va  | v_s|      \)^2 ds  \right]   \nonumber\\
& \leq &C_3      \int_0^t    \EE  \left[  \sup \limits_{0\leq s\leq t }      |   X^{ \bar  u+\va v } _s  -X^ {\bar  u } _s     |^2 \right]     ds   
 +\va ^2 \,  \E \int_0^T      |v_s|^2ds.\nonumber 
\end{eqnarray*}
 The desired inequality now follows from the Gronwall's inequality.
\qed 

\medskip

We first introduce the following variational equations for an $\R^n$-valued process 
$V _t:=V^{  \bar u, v}_t$ with $\ t\in[0,T] $.

\begin{eqnarray}\label{e:V}
\left\{\begin{aligned} d V _t  =
& \EE^\prime  \left[  \nabla_x b (t,  \Theta^{\bar u}_t )   V_t
+ \nabla_y b(t,  \Theta^{\bar u}_t ) (V_t )^\prime 
  +\nabla_u  b (t,  \Theta^{\bar u}_t )  v_t  \right] dt \\
 &     +  \EE^\prime  \left[  \nabla_x   \delta  (t, \Theta^{\bar u}_t ) V_t  
+  \nabla_y  \delta  (t,  \Theta^{\bar u}_t )  \cdot (V_t )^\prime 
    +\nabla_u b (t,  \Theta^{\bar u}_t )  v_t   \right] dL_t  \\
  &+  \EE^\prime  \left[    \nabla_x    \sigma(t,  \Theta^{\bar u}_t ) V_t    
+   \nabla_y  \sigma(t,   \Theta^{\bar u}_t )   (V_t )^\prime 
   +\nabla_u   \sigma (t,  \Theta^{\bar u}_t ) v_t  \right]      dB_{L_ t},   \\
V_0  = &  0 .\\
\end{aligned}
\right.
\end{eqnarray}
 
 Under Hypothesis \ref{HP1}, Hypothesis \ref{HP0-SDEs} holds. So we know from Theorem \ref{exis-unique-SDEs},
 that 
  the equation  MF-SDE  \eqref{SDE} together with \eqref{e:V} combined has a unique $L^2$-strong solution
  $(X^{\bar u}, V)  \in \sM [0, T]$ (as an $\R^{n+n}$-valued $\bF'$-progressively measurable process).
   It   has the property that 
 \begin{equation}\label{e:Vb}
 \EE  \left[ \sup\limits_{0\leq t\leq T  } | V _t|^2 \right]  < \infty.
 \end{equation} 
  
  \bigskip

\begin{lemma}\label{xva2}
 Suppose the Hypothesis  \ref{HP1}  holds and let
\[
\wt{X}^{\va}_t :=\frac{X^{\bar u+\va v}_t-X^{{\bar u} }_t}{\va }-V_t.\quad    \]
Then there is a constant $C>0$ so that  
$$
  \E  \left[ \sup\limits_{0\leq t\leq T}  |  \wt{X}^\va_t|^2 \right]  \leq C \,  \eps^2
  \quad \hbox{for } \eps \in (0, 1), 
$$
 \end{lemma}

 \bigskip

\noindent{\bf Proof.}     It follows from (\ref{e:SDE})  
that $\wt X^\eps_0 =0$ and
\begin{eqnarray*}
  d\wt{X}^\va_t
  &=&\bigg(\eps^{-1}  \E'  \int_0^1 \frac{d}{d\lambda} 
  \Big(  b(t, X^{\bar u}_t + \lambda \eps (V_t + \wt X^\eps_t), (X^{\bar u+\eps v}_t)', \bar u_t + \eps v_t)  
   \\
  && \qquad + b(t, X^{\bar u }_t, (X^{\bar u}_t )' + \lambda \eps (V_t + \wt X^\eps_t)',  \bar u_t + \eps v_t)  
  + b(t, X^{\bar u }_t, (X^{\bar u }_t)', \bar u_t +  \lambda \eps v_t) \Big) d\lambda   \\
  && \quad -\E' \Big[ \nabla_x b (t, \Theta^{\bar u}_t) V_t + \nabla_y b (t, \Theta^{\bar u}_t) (V_t)'   
   + \nabla_u b (t, \Theta^{\bar u}_t) v_t \Big] \bigg) dt \\
   && + \bigg( \eps^{-1} \E'  \int_0^1 \frac{d}{d\lambda} 
   \Big(  \delta (t, X^{\bar u}_t + \lambda \eps (V_t + \wt X^\eps_t), (X^{\bar u+\eps v}_t)', \bar u_t + \eps v_t)  
   \\
  && \qquad + \delta (t, X^{\bar u }_t, (X^{\bar u}_t )' + \lambda \eps (V_t + \wt X^\eps_t)',  \bar u_t + \eps v_t)  
  + \delta (t, X^{\bar u }_t, (X^{\bar u }_t)', \bar u_t +  \lambda \eps v_t) \Big) d\lambda   \\
  && \quad -\E' \Big[ \nabla_x \delta (t, \Theta^{\bar u}_t) V_t + \nabla_y \delta (t, \Theta^{\bar u}_t) (V_t)'   
   + \nabla_u \delta (t, \Theta^{\bar u}_t) v_t \Big] \bigg) dL_t \\
  && + \bigg(  \eps^{-1} \E'  \int_0^1 \frac{d}{d\lambda} 
   \Big(  \sigma (t, X^{\bar u}_t + \lambda \eps (V_t + \wt X^\eps_t), (X^{\bar u+\eps v}_t)', \bar u_t + \eps v_t)  
   \\
  && \qquad + \sigma (t, X^{\bar u }_t, (X^{\bar u}_t )' + \lambda \eps (V_t + \wt X^\eps_t)',  \bar u_t + \eps v_t)  
  + \sigma (t, X^{\bar u }_t, (X^{\bar u }_t)', \bar u_t +  \lambda \eps v_t) \Big) d\lambda   \\
  && \quad -\E' \Big[ \nabla_x \sigma (t, \Theta^{\bar u}_t) V_t + \nabla_y \sigma (t, \Theta^{\bar u}_t) (V_t)'   
   + \nabla_u \sigma (t, \Theta^{\bar u}_t) v_t \Big] \bigg) dB_{L_t} \\
 \end{eqnarray*}
By Burkholder-Davis-Gundy;s inequality and \eqref{e:2.3}, for $t\in [0, T]$, 
\begin{eqnarray} \label{e:6.4} 
&& {\EE}  \left[ \sup_{0\leq s\leq t}  |\wt X^\eps_s|^2  \right].  \nonumber \\
&\leq & C_1\bar  \E  \int_0^t \Big( \eps^{-1} \int_0^1 
 \frac{d}{d\lambda} 
  \Big(  b(s, X^{\bar u}_s + \lambda \eps (V_s + \wt X^\eps_s), (X^{\bar u+\eps v}_s)', \bar u_s + \eps v_s)  
   \nonumber \\
  && \qquad + b(s, X^{\bar u }_s, (X^{\bar u}_s)' + \lambda \eps (V_s + \wt X^\eps_s)',  \bar u_s + \eps v_s)  
  + b(s, X^{\bar u }_s, (X^{\bar u }_s)', \bar u_s +  \lambda \eps v_s)   \Big) d\lambda  \nonumber \\
  && \qquad - \left( \nabla_x b (s, \Theta^{\bar u}_s) V_s + \nabla_y b (s, \Theta^{\bar u}_s) (V_s)'   
   + \nabla_u b (s, \Theta^{\bar u}_s) v_s \right)  \bigg)^2  ds  \nonumber \\
   && +  
   C_1\bar  \E  \int_0^t \Big( \eps^{-1}  \int_0^1 
 \frac{d}{d\lambda} 
  \Big(  \delta(t, X^{\bar u}_s + \lambda \eps (V_s + \wt X^\eps_s), (X^{\bar u+\eps v}_s)', \bar u_s + \eps v_s)  
   \nonumber \\
  && \qquad + \delta (s, X^{\bar u }_s, (X^{\bar u}_s )' + \lambda \eps (V_s+ \wt X^\eps_s)',  \bar u_s + \eps v_s)  
  + \delta (s, X^{\bar u }_s, (X^{\bar u }_s)', \bar u_s +  \lambda \eps v_s)  \Big) d\lambda    \nonumber \\
  && \qquad - \left( \nabla_x \delta (t, \Theta^{\bar u}_s) V_s + \nabla_y \delta (s, \Theta^{\bar u}_s) (V_s)'   
   + \nabla_u \delta (s, \Theta^{\bar u}_s) v_t   \right)  \bigg)^2  ds  \nonumber \\
     && +  
   C_1\bar  \E  \int_0^t \Big(  \eps^{-1} \int_0^1 
 \frac{d}{d\lambda} 
  \Big(  \sigma(s, X^{\bar u}_s + \lambda \eps (V_s + \wt X^\eps_s), (X^{\bar u+\eps v}_s)', \bar u_s + \eps v_s)  
   \nonumber \\
  && \qquad + \sigma (s, X^{\bar u }_s, (X^{\bar u}_s )' + \lambda \eps (V_s + \wt X^\eps_s)',  \bar u_s+ \eps v_s)  
  + \sigma (s, X^{\bar u }_s, (X^{\bar u }_s)', \bar u_s +  \lambda \eps v_s) \Big) d\lambda     \nonumber \\
  && \qquad - \left( \nabla_x \sigma (s, \Theta^{\bar u}_s) V_s + \nabla_y \sigma (t, \Theta^{\bar u}_s) (V_t)'   
   + \nabla_u \sigma (s, \Theta^{\bar u}_s) v_s  \right)  \bigg)^2  ds .
\end{eqnarray}
Each of these three terms can be estimated as follows. With  $\varphi = b, \delta, \sigma$, 
by the bounded and Lipschitz continuity of the first order partial derivatives of $\varphi$ as well as   
and Lemma \ref{xpi-x} and  \eqref{e:Vb}, 
\begin{eqnarray*}   
&& \bar  \E  \int_0^t \Big( \eps^{-1} \int_0^1 
 \frac{d}{d\lambda} 
  \Big(  \varphi (s, X^{\bar u}_s + \lambda \eps (V_s + \wt X^\eps_s), (X^{\bar u+\eps v}_s)', \bar u_s + \eps v_s)  
   \nonumber \\
  && \qquad + \varphi (s, X^{\bar u }_s, (X^{\bar u}_s)' + \lambda \eps (V_s + \wt X^\eps_s)',  \bar u_s + \eps v_s)  
  + \varphi (t, X^{\bar u }_s, (X^{\bar u }_s)', \bar u_s +  \lambda \eps v_s) \Big)  d\lambda   \nonumber \\
  && \qquad -  \left( \nabla_x \varphi  (s, \Theta^{\bar u}_s) V_s + \nabla_y \varphi  (s, \Theta^{\bar u}_s) (V_s)'   
   + \nabla_u \varphi  (s, \Theta^{\bar u}_s) v_s \right) d\lambda \bigg)^2  ds  \nonumber \\
&=&    \bar  \E  \int_0^t \Big( \int_0^1    
  \Big(  \nabla_x \varphi (s, X^{\bar u}_s + \lambda \eps (V_s + \wt X^\eps_s), (X^{\bar u+\eps v}_s)', \bar u_s + \eps v_s) 
   (V_s + \wt X^\eps_s)  
   \nonumber \\
  && \qquad + \nabla_y \varphi (s, X^{\bar u }_s, (X^{\bar u}_s )' + \lambda \eps (V_s + \wt X^\eps_s)',  \bar u_s + \eps v_s)  
   ((V_s)'    + ( \wt X^\eps_s)')  \\ 
    && \qquad   
  + \nabla_u \varphi (s, X^{\bar u }_s, (X^{\bar u }_s)', \bar u_s +  \lambda \eps v_s)   v_s  
 -   \nabla_x \varphi  (s, \Theta^{\bar u}_s) V_s -  \nabla_y \varphi  (s, \Theta^{\bar u}_s) (V_s)'    \\
&& \qquad    - \nabla_u \varphi  (s, \Theta^{\bar u}_s) v_s \Big)  d\lambda \Big)^2  ds  \nonumber \\
&\leq & C_2 \E \int_0^t | \wt X^\eps_s|^2 ds + C_2 \eps^2  \E \int_0^T ( |V_s|^2 + |v_s|^2) ds \\
&\leq & C_2   \int_0^t \E \Big[  \sup_{r\in [0, s]} | \wt X^\eps_r|^2 \Big] ds + C_2 \eps^2  \E \int_0^T ( |V_s|^2 + |v_s|^2) ds .
\end{eqnarray*}
We thus have by \eqref{e:6.4} that 
$$
\E  \left[ \sup_{0\leq s\leq t}  |\wt X^\eps_s|^2  \right]
\leq 3C_2   \int_0^t \E \Big[  \sup_{r\in [0, s]} | \wt X^\eps_r|^2 \Big] ds + 3C_2 \eps^2  \E \int_0^T ( |V_s|^2 + |v_s|^2) ds .
$$
  The desired conclusion now follows from  the Gronwall's inequality.
\qed

\bigskip

\begin{hypothesis}\label{H:6.5}  \rm 
The drivers of the modified  MF-BSDE to be considered in this section are functions
$$
h_i=h_i (\bar \omega,  t, y, z, \wt y, \wt z):  \bar \Omega   \times [0, T] 
\times \R^n \times \R^{n\times d}  \times \R^n \times \R^{n\times d} \to \R^n \quad \hbox{for } i=1, 2,
$$
which are  $\wt \bF$-progressively measurable for all $(y, z, \wt y, \wt z)$ and satisfies 
the following assumptions:
\begin{enumerate}
\item[\rm (i)] $\bar \E \int_0^T | h_1 (\bar \omega, t, {\bf 0} )|^2 dt 
+ \bar \E \int_0^T | h_2 (\bar \omega, t, {\bf 0} )|^2 dL_t < \infty,$
where ${\bf 0}$ denotes the origin of $\R^n \times \R^{n\times d}  \times \R^n \times \R^{n\times d}.$

\item[\rm (ii)] There is a constant $C_0>0$ so that $\bar \bP$-a.s. $\bar \omega \in \bar \Omega$, 
for all $t\in [0, T]$,
$y_i, \wt y_i \in \R^n$, $ z_i, \wt z_i \in \R^{n\times d}$ with  $i=1, 2$, 
$$
|h_k(\bar \omega, t, y_1, z_1, \wt y_1, \wt z_1)- h_k(\bar \omega, t, y_2, z_2, \wt y_2, \wt z_2)|
\leq C_0 ( |y_1-y_2| + |z_1-z_2| + |y_1'-y_2' | + |z_1'-z_2' | )
$$
 for $k=1, 2.$
\end{enumerate}
\end{hypothesis}

For notational simplicity, we will typically drop $\bar \omega$ from the expressions  of the above random processes or variables. 

\begin{theorem}\label{T:BSDE2}
Suppose Hypothesis \ref{H:6.5} holds. For any given $\xi \in  L^2 (\wt\sF_T;\R^n)$,
  the following  MF-BSDE 
 \begin{eqnarray}\label{e:6.2}
 d Y_t   =\E' \left[ h_1(\bar \omega, t,Y_t, Z_t, (Y_t)^\pr, (Z_t)^\pr  ) \right] dt 
  + \E' \left[  h_2 (\bar \omega, t,Y_t, Z_t, (Y_t)^\pr, (Z_t)^\pr  ) \right] dL_t +Z_t d B_{L_t} 
 \end{eqnarray} 
   having  $Y_T  = \xi$
   admits unique $L^2$ adapted solution $(Y  , Z  ) \in \sM^2 [0,T]$
    in the sense that if $(\wt Y  , \wt Z  ) \in \sM^2 [0,T]$ is another solution of \eqref{e:BSDE2},
  then $\wt Y_t=Y_t$ for all $t\in [0, T]$ with probability one and $\E \int_0^T |Z_s -\wt Z_s|^2 dL_ s    =0$.
\end{theorem}

The proof of the above theorem is similar to that for Theorem \ref{T:BSDE} so it is omitted here. 

\medskip

In view of \eqref{e:2.3}, we can equivalently rewrite the  MF-BSDE \eqref{e:6.2} as 
\begin{eqnarray}\label{e:6.3}
 d Y_t   &=& \E' \left[ h_1(\bar \omega, t,Y_t, Z_t, (Y_t)^\pr, (Z_t)^\pr  ) \right] dt 
  + \kappa^{-1} \1_{\{R_t=0\}} \E' \left[  h_2 (\bar \omega, t,Y_t, Z_t, (Y_t)^\pr, (Z_t)^\pr  ) \right] dt  \nonumber  \\
&&    +Z_t d B_{L_t} 
 \end{eqnarray} 
   having  $Y_T  = \xi$.

 \bigskip

Define two  Hamiltonians 
\begin{equation}\label{H1}
H(t, x, y,  u, p):= f(t, x,  y, u)  - b(t, x,  y, u) \cdot p   
\end{equation}
and 
\begin{equation}\label{H2}
\sH(t, x, y,  u, p, q):= g(t, x,  y, u)  - \delta(t, x,  y, u)\cdot p - \sigma    (t, x,  y, u)\cdot  q,
\end{equation}
where $t\geq 0$, $x, y, u, p\in \R^n$ and $q\in \R^{n\times d}$.  
We use $\nabla_x H$,  $\nabla_y H$ and $\nabla H_u$ to denote the gradient of $H$ with respect to $x$, $y$ and $u$,
 respectively;  that is,
\begin{eqnarray*}
\nabla_x H(t, x, y,  u, p)&:=& \nabla_x f(t, x,  y, u)  -  \nabla_x ( b(t, x,  y, u) \cdot p) , \\
\nabla_y H(t, x, y,  u, p)&:=& \nabla_y f(t, x,  y, u)  -  \nabla_y( b(t, x,  y, u) \cdot p) , \\
\nabla_u H(t, x, y,  u, p)&:=& \nabla_u f(t, x,  y, u)  - \nabla_x ( b(t, x,  y, u) \cdot p),
\end{eqnarray*}
 Similar notions apply to $\sH$ as well.

\medskip

Recall the notation of $\Theta^u_t$ and $\wh \Theta^u_t$ from    \eqref{e:Theta}.
Using the above two Hamiltonians,  the  MF-SDE \eqref{e:SDE} for $X=X^u$ can be written as 
\begin{equation}\label{e:SDE2}
\left\{\begin{aligned}
 dX_t =&\ -  \E ^\prime \left[ \nabla_p H (t,   \Theta^u_t ) \right] dt  
 - \E ^\prime \left[\nabla_p \sH_p (t,   \Theta^u_t) \right] dL_t
- \E^ \prime \left[ \nabla_q \sH (  t,  \Theta^u_t  ) \right] dB_{L_t}  ,   \\
 X_0=&\ x . 
\end{aligned}
\right.
\end{equation}

 \medskip

Fix some  $\bar u \in \cU' [0, T]$. 
We consider  the following associated adjoint equation for $(p_t, q_t)$ taking values in $\R^n \times \R^{n\times d}$:  
 \begin{equation}\label{BSDE-p}
\left\{\begin{aligned}
d p_t =&\    \EE^\prime\big[ \nabla_x H(t,  \Theta^{\bar u}_t,  p_t) + 
 \nabla_y H (t,    \wh \Theta^{\bar u}_t ,  (p_t)^\prime ) \big]dt  \\
   &  + \kappa^{-1} \EE ^\prime \big [ \nabla_x \sH(t,  \Theta^{\bar u}_t,  p_t, q_t) 
    \big]  \1_{\{R_t=0\}}  dt \\
&   + \kappa^{-1} \EE ^\prime \big [   \nabla_y \sH (t,    \wh \Theta^{\bar u}_t ,  (p_t)^\prime, (q_t)^\prime )  \1_{\{ (R_t)^\pr =0\}}  \big]dt  
   +q_t  dB_{L_t}  , \\  
 p_T  = &\  -\EE^\pr \left[  \nabla_x h ( X^{\bar u}_T , (X^{\bar u}_T)^\prime ) +    \nabla_y h ( (X^{\bar u}_T)^\prime , X^{\bar u}_T  ) \right]. 
\end{aligned}
\right.
\end{equation}

\begin{theorem}
Suppose Hypothesis \ref{HP1} holds. 
The  MF-BSDE  \eqref{BSDE-p}
   admits unique $L^2$ adapted solution $(p_t, q_t) \in \sM^2 [0,T]$
    in the sense that if $(\bar  p_t,, \bar q_t) \in \sM^2 [0,T]$ is another solution of \eqref{BSDE-p},
  then $\bar  p_t,=p_t$ for all $t\in [0, T]$ with probability one and $\E \int_0^T |q_s -\wt q_s|^2 dL_ s    =0$.
\end{theorem}

\pf  Equation \eqref{BSDE-p}  can be identified with  MF-BSDE \eqref{e:6.2} with
  \begin{eqnarray}\label{e:h12}
\left\{\begin{aligned} 
 h_1 (\omega, \omega',  t, p, q, \wt p, \wt q)
:= \, &   \nabla_x f(  t,   X^{\bar u}_t (\omega), X^{\bar u}_t (\omega') ,  {\bar u} _t (\omega)  )   
   + \nabla_y f(  t,   X^{\bar u}_t (\omega') ,    X^{\bar u}_t (\omega),  {\bar u}_t (\omega'))  \\
&   -  \nabla_x  \left( b(t,  X^{\bar u}_t (\omega),    X^{\bar u}_t (\omega'),   {\bar u}_t (\omega)) \cdot  p \right) 
   -   \nabla_y  \left( b(t,     X^{\bar u}_t (\omega'),   X^{\bar u}_t (\omega) ,   {\bar u}_t (\omega')) \cdot  \wt p \right)  \\
 & +  \kappa^{-1}  \1_{\{ R_t (\omega')= 0\}}  \nabla_y g  (t, X^{\bar u}_t (\omega'),   X^{\bar u}_t (\omega) ,   {\bar u}_t (\omega'))    \\
 & - \kappa^{-1}   \1_{\{ R_t (\omega')= 0\}}  \nabla_y \left( \delta( t, (  X^{\bar u}_t (\omega'),   X^{\bar u}_t (\omega) , 
  {\bar u}_t (\omega'))    \cdot  \wt p \right)  \\
&  - \kappa^{-1}  \1_{\{ R_t (\omega')= 0\}}    \nabla_y \left( \sigma ( t,   X^{\bar u}_t (\omega'),   X^{\bar u}_t (\omega) , 
     {\bar u}_t (\omega'))     \cdot  \wt q \right), \\
  h_2 (\omega, \omega',  t, p, q, \wt p, \wt q)
 := \, &  \nabla_x g ( t,   X^{\bar u}_t (\omega), X^{\bar u}_t (\omega') ,  {\bar u} _t (\omega)  )   
  -     \nabla_x  \left( \delta  ( t,   X^{\bar u}_t (\omega), X^{\bar u}_t (\omega') ,  {\bar u} _t (\omega)  )  \cdot p \right) \\
 &-   \nabla_x \left( \sigma ( t,   X^{\bar u}_t (\omega), X^{\bar u}_t (\omega') ,  {\bar u} _t (\omega)  )  \cdot  q  \right)  ,  \\
 \xi := \, &  -  \EE^\pr \left[ \nabla_x h ( X^{\bar u}_T   , (X^{\bar u}_T)^\prime ) +    \nabla_y h ( (X^{\bar u}_T)^\prime , X^{\bar u}_T  ) \right].
\end{aligned}
\right.
\end{eqnarray}

\medskip

 Under Hypothesis \ref{HP1}, Hypothesis \ref{H:6.5} holds for $(h_1, h_2)$. 
 Thus  by Theorem \ref{T:BSDE2}, 
  MF-BSDE \eqref{BSDE-p}  has a unique $L^2$ adapted solution
 $(p, q)\in \sM^2 [0, T]$. \qed
 
 \bigskip

\begin{theorem}\label{Necessaryconv}
Assume that the Hypothesis  \ref{HP1}    holds. Suppose that  $\bar u \in \cU'[0, T]$   is a local optimal
control  of \eqref{e:SDE}. Denote by  $   \bar x  $  its corresponding  state process.  
Then for every $u\in  \cU'[0, T]$, 
  \begin{equation}\label{e:SMP}
 \E \left[ \int_0^T \E^\pr \left[ \nabla_u H (t, \Theta^{\bar u}_t,  p_t)  
      +       \1_{\{R_t=0\} }\kappa^{-1}
    \nabla_u  \sH  (t, \Theta^{\bar u}_t, p_t, q_t) \right] \cdot  (u_t-\bar u_t) dt \right] \geq 0.
   \end{equation} 
Moreover, if this local optimal control   $\bar u$   is an interior point of $\cU' [0, T]$,
then $d \bP\times dt$ a.s. on   $\Omega \times [0,T]$,  
 \begin{equation}\label{e:SMP1}
   \EE^\pr \big[ \nabla_u H  (t, \Theta^{\bar u}_t, \bar  p_t, ) \big]
    +  \1_{ \{R_t=0\}}    \kappa^{-1}   \EE^\pr \big[ \nabla_u \sH (t,  \Theta^{\bar u}_t,  \bar  p_t, \bar q_t  ) ]    =0,
   \end{equation}
where $\bar  p_t,:=  \EE[  p_t|   \sF^{\prime}_t  ]$ and $\bar q_t:=  \EE[ q_t |  \sF^{\prime}_t    ] $.

   \end{theorem}

\noindent{\bf Proof.}
For  $u\in  \cU' [0, T]$,  set $v=u- \bar u $. Then  for any $\eps \in (0, 1)$,
$\bar u + \eps v = (1-\eps ) \bar u + \eps u\in \cU' [0, T]$ and so 
\begin{eqnarray}\label{dcostfunc}
 0   &\leq & \lim_{\va \to   0} \frac{ J(\bar u  +\va v )- J(\bar u  )     }{\va}
 \nonumber \\
&=&  \lim_{\va \to   0} \frac 1 \va  \EE   \Big [    \int_0^T  \EE^\prime \big[  f(t,  \Theta^{\bar u + \eps v}_t   ) 
-f(t, \Theta^{\bar u}_t)\big] dt  
+  \int_0^T  \EE^\prime \big[  g(t,      \Theta^{\bar u + \eps v}_t   ) 
-g(t, \Theta^{\bar u}_t)\big] dL_t \Big] \nonumber  \\
&& \hskip 0.6truein   
 +\E  \EE^\prime \left[ h(X^{\bar u+\va v }_T, (X^{\bar u+\va v }_T)^\prime   )  -  h(  X^{\bar u}_T, (X^{\bar u}_T)^\prime ) \right]      \nonumber \\
 &= &  {\mathbb{E}} \bigg [  \int_0^T  \EE^\prime   \left[ \nabla_x f (t, \Theta^{\bar u}_t) \cdot V_t
     + \nabla_y f(t, \Theta^{\bar u}_t) \cdot (V_t)^\prime  
      +\nabla_u f(t, \Theta^{\bar u}_t ) \cdot v_t \right] dt       \nonumber \\
    && \qquad + \int_0^T  \EE^\prime   \left[ \nabla_x g (t, \Theta^{\bar u}_t) \cdot V_t
     + \nabla_y g(t, \Theta^{\bar u}_t) \cdot (V_t)^\prime 
       +\nabla_u g(t, \Theta^{\bar u}_t ) \cdot v_t \right] dL_t  \bigg]
      \nonumber \\
&&\qquad + \E \EE^\pr \left[
\nabla_x h  \left(X^{\bar u  }_T, (X^{\bar u  }_T)^\prime   \right) \cdot V_T  
 +  \nabla_y h \left(  X^{\bar u}_T, (X^{\bar u}_T)^\prime \right) \cdot ( V_T)^\prime  \right]    ,
\end{eqnarray}
where  $V_t$ is the $\R^n$-valued process  given by \eqref{e:V}, 
which by Lemma \ref{xva2} is the derivative process of $X_t^{\bar u + \eps v}$
in $\eps$ at $\eps=0$. 
By It\^o's formula,  we have
\begin{eqnarray*}
 &&   -\EE \left  [V_T \cdot \EE^\pr \left[     \nabla_x h  (X^{\bar u  }_T, (X^{\bar u  }_T)^\prime   )  + \nabla_y h (  X^{\bar u}_T, (X^{\bar u}_T)^\prime )         \right]  \right]  \nonumber \\
&=& \EE   [V_T \cdot p_T - V_0 \cdot p_0]  
= \EE \left[ \int_0^T V_t \cdot dp_t + \int_0^T p_t \cdot dV_t + \<V, p\>_T \right]  \nonumber \\ 
& =&   \EE   \int_0^T  \EE^\pr   \left [  V_t \cdot  \left((   \nabla_x f (t,  \Theta^{\bar u}_t  )  
 + \nabla_y   f  (  t,   \wh \Theta^{\bar u }_t   )   \right) \right]dt  \nonumber \\
 && \qquad + \kappa^{-1}  \EE^\pr   \left[  V_t \cdot \left(   \nabla_x g (t,  \Theta^{\bar u}_t  )  
 \1_{\{R_t=0\}} + \nabla_y   g   (  t,   \wh \Theta^{\bar u }_t   )   \1_{\{ (R_t)^\pr =0\}}  \right) \right]dt   \nonumber \\
&&+ \EE  \int_0^T\EE^\pr    \left[ p_t  \cdot  \nabla_u b (   t,  \Theta^{\bar u}_t   ) v_t    \right]  dt    \nonumber \\
&& +\EE  \int_0^T \EE^\pr   \left[   \nabla_u \delta ( t,  \Theta^{\bar u}_t    )  p_t 
+    \nabla_u \sigma ( t,  \Theta^{\bar u}_t    )   q_t   \right]  v_t dL  _ {(t-a) ^+ } .
\end{eqnarray*}
This together with  \eqref{dcostfunc} yields   that 
 \begin{eqnarray*}
0 &\leq &  \E \Big[ \int_0^T  \E^\pr \left [ \left(  \nabla_u f(t, \Theta^{\bar u}_t ) 
-  \nabla_u b (t, \Theta^{\bar u}_t ) p_t  \right) v_t \right] dt
      \nonumber \\
    && \qquad + \int_0^T  \EE^\prime   \left[ \big(  \nabla_u g (t, \Theta^{\bar u}_t)
    - \nabla_u \delta (t, \Theta^{\bar u}_t) p_t 
     - \nabla_u \sigma (t, \Theta^{\bar u}_t)  q_t \big) v_t \right] dL_t \Big] 
      \nonumber \\
      &=& \E \left[ \int_0^T \E^\pr \left[ H_u (t, \Theta^{\bar u}_t, p_t )
      +       \1_{\{R_t=0\} }\kappa^{-1}
      \sH_u (t, \Theta^{\bar u}_t, p_t, q_t) \right]  (u_t-\bar u_t) dt \right], 
 \end{eqnarray*}
 where we used the identities 
 \begin{eqnarray*}
\E  \EE^\pr \left[ \nabla_y h (  X^{\bar u}_T, (X^{\bar u}_T)^\prime )  ( V_T)^\prime  \right]   
&=& \E \left[ \EE^\pr \left[ \nabla_y h (   (X^{\bar u}_T)^\prime ,  X^{\bar u}_T)  \right] V_T  \right]  ,   
\\
  {\mathbb{E}}  \int_0^T   \E^\pr \left[  \nabla_y f(t, \Theta^{\bar u}_t) (V_t)^\prime     \right] dt
 & =&  {\mathbb{E}}  \int_0^T   \E^\pr \left[  \nabla_y f(t,   \wh \Theta^{\bar u}_t)      \right] 
 V_t dt,  
 \\
  {\mathbb{E}}  \int_0^T   \E^\pr \left[ p_t  \cdot  \nabla_y b(t, \Theta^{\bar u}_t) (V_t)^\prime     \right] dt
 & =&  {\mathbb{E}}   \int_0^T   \E^\pr \left[ (p_t)^\pr  \cdot \nabla_y b(t,   \wh \Theta^{\bar u}_t)      \right] 
 V_t dt, 
  \\
 \E  \int_0^T  \E^\pr \left[ \nabla_y g(t, \Theta^{\bar u}_t) (V_t)^\prime    \right] dL_t
 &=& \kappa^{-1}  \E  \int_0^T  \E^\pr \left[ \nabla_y g(t, \Theta^{\bar u}_t) (V_t)^\prime    \right] 
 \1_{\{R_t=0\}} dt \\
 &=& \kappa^{-1}  \E  \int_0^T  \E^\pr \left[ \nabla_y g(t,   \wh \Theta^{\bar u}_t ) 
 \1_{\{ (R_t)^\pr =0\}}   \right] V_t dt  \\
 \E  \int_0^T p_t  \cdot  \E^\pr \left[ \nabla_y \delta (t, \Theta^{\bar u}_t) (V_t)^\prime    \right] dL_t
   &=& \kappa^{-1}  \E  \int_0^T  \E^\pr \left[ (p_t)^\pr  \cdot \nabla_y \delta (t,   \wh \Theta^{\bar u}_t ) 
 \1_{\{ (R_t)^\pr =0\}}   \right] V_t dt  \\
 \E  \int_0^T q_t \cdot  \E^\pr \left[ \nabla_y \sigma (t, \Theta^{\bar u}_t) (V_t)^\prime    \right] dL_t
   &=& \kappa^{-1}  \E  \int_0^T   \E^\pr \left[ (q_t)^\pr \cdot \nabla_y \sigma (t,   \wh \Theta^{\bar u}_t ) 
 \1_{\{ (R_t)^\pr =0\}}   \right] V_t dt ,
 \end{eqnarray*}
due to the fact that   $( (R_t)^\pr , (X^{\bar u}_T)^\pr, (V_t)^\pr, (u_t)^\pr, (p_t)^\pr, (q_t)^\pr)$ 
are independent copy of  $( R_t, X^{\bar u}_T, V_t, u_t, p_t, q_t)$ on $\bar \Omega = \Omega \times \Omega$. 
This establishes \eqref{e:SMP}. 

Since both $H_u$ and $\sH_u$ are linear in $p$ and $q$ and $X^{\bar u}_T$, $\1_{R_t=0\}}$, $\bar u_t$ and $u_t$ 
are $\{\sF'_t\}$-measurable, we have by \eqref{e:SMP} that 
\begin{equation}\label{e:6.9}
 \E \left[ \int_0^T \E^\pr \left[ \nabla_u H (t, \Theta^{\bar u}_t, \bar  p_t, )
      +       \1_{\{R_t=0\} }\kappa^{-1}  
      \nabla_u  \sH (t, \Theta^{\bar u}_t, \bar  p_t, \bar q_t) \right] \cdot (u_t-\bar u_t) dt \right]
      \geq 0
\end{equation}
for every $u\in \cU'[0, T]$.
When the local optimal control $\bar u$ is an interior point of $\cU' [0, T]$, 
 we conclude from \eqref{e:6.9} that \eqref {e:SMP1} holds. 
\qed

\section{Sufficient conditions for maximum principle}\label{S:7}

In this section,   let $0<T<\infty$, and $x_0\in \R$.
 Let $\bar u  $ be an admissible control in  $\cU' [0, T]$ and   $X^{\bar u} $   be
  the corresponding solution to the  \eqref{e:SDE} with $\bar u$ in place of $u$.  
Let  $(p  , q  ) $ be the solution to  the adjoint BSDE equation \eqref{BSDE-p} 
associated with $(\bar u, X^{\bar u})$. 
 
 Recall the two Hamiltonians $H$ and $\sH$ defined in \eqref{H1} and \eqref{H2}
and, for $u\in \cU' [0, T]$, the processes $\Theta^u$ and $\wh \Theta^u$ in \eqref{e:Theta}.
For notational simplicity,   for $\bar u, u\in \cU' [0, T]$, set  
 $$ H(t, u):= H  (t,   \Theta^{u}_t, p_t ), \quad \sH (t, u):= \sH  (t,   \Theta^{u}_t, p_t, q_t ),
 $$
 $$
 \wh H(t, u):= H  (t,   \wh \Theta^{u}_t, (p_t)'  ), \quad \wh \sH (t, u):= \sH  (t,   \wh \Theta^{u}_t, (p_t)',  (q_t)' ),
 $$
 \begin{equation}\label{e:H*}
 H^* (t, u) := H(t, u)+ \kappa^{-1}  \1_{\{R_t=0\}} \sH (t, u), \quad 
 \wh  H^* (t, u) := \wh H(t, u)+ \kappa^{-1} \1_{\{ (R_t)'=0\}} \wh \sH(t, u),
 \end{equation}
  where $(p_t, q_t)$ is the solution for the adjoint equation \eqref{BSDE-p} corresponding to $\bar u \in \cU ' [0, T]$.

\begin{theorem}[Sufficient Stochastic Maximum Principle]\label{Sufficiency} 
 
Suppose that  Hypothesis \ref{HP1} holds and $\bar u\in \cU' [0, T]$. 
  With the above notations,    assume  that     
  \begin{itemize}
 \item [\rm (i)]  $h $ is  convex  functions on $\R^n\times \R^n$, 
 
  \item [\rm (ii)]  for every $u\in \cU' [0, T]$, 
    \begin{eqnarray*}
 \E \left[ \int_0^T \E^\prime \left[ H ^* (t, u)- H^* (t, \bar u) \right] dt \right] 
\geq \E \Big [ \int_0^T (X^u_t - X^{\bar u}_t)  \cdot  \E' \big[ \nabla_x H^*  (t, \bar u) +\nabla_y \wh H^* (t, \bar u)   \big] dt   \Big] .
  \end{eqnarray*} 
      \end{itemize}
Then $\bar u (\cdot )$ is an optimal control. 
\end{theorem}

\pf    For $u\in \cU_a' [0, T]$, 
 \begin{eqnarray}\label{J-Jbar}
J(u)-J( \bar u)  
&=&\EE \left[       \int_0^T \EE^\pr \left[  f( t, \Theta^u_t)  -  f( t, \Theta^{\bar u}_t) \right] dt  
+ \int_0^T \E'  \left[  g( t, \Theta^u_t)  -  g( t, \Theta^{\bar u}_t) \right] dL_t   \right] 
  \nonumber\\
&&  + \,           \EE \left[  \EE^\pr\left[ h(X^{ u}_T, (X^{ u}_T) ^\pr  ) - h( X^{\bar u} _T, (X^{\bar u}_T)^\pr ) \right] \right]. 
\end{eqnarray}
By the convexity of $h$  and It\^o's formula, 
\begin{eqnarray}\label{p-6-1}
&&  \EE \left[ \EE^\pr \left[ h( X^{ u}_T , (       X^{ u}_T) ^\pr   ) - h(  X^{\bar u}_T , (X^{\bar u}_T)^\pr ) \right] \right]     \nonumber\\
  & \ge& \EE \left[ \EE^\pr \left[   \nabla_x h (  X^{ \bar u}_T,      (X^{\bar  u}_T) ^\pr    )   \cdot (  X^u_T    -   X^{\bar u}_T  ) 
  +   \nabla_y h (    X^{\bar u}_T,  (X^{\bar u}_T )^\pr  )  \cdot ( (X^u_T)^\pr   -  (X^{\bar u}_T)^\pr )    \right] \right]  \nonumber\\
&=&-\E \left[  \E' \left[   p_T  \cdot (  X^u_T    -   X^{\bar u}_T  )   -p_0  \cdot (  X^u_0   -   \bar X_0  )  \right ]\right]   \nonumber\\
&=& - \EE\bigg [  \E' \int_0^T p_t \cdot d ( X^u_t - X^{\bar u}_t  ) + \int_0^T ( X^u_t - X^{\bar u}_t  ) \cdot dp_t + \< p, X^u_t - \bar X\>_T   \bigg]
     \nonumber\\
  &=& - \EE\bigg[  \E' \int_0^T  p_t \cdot \left( b (t, \Theta^u_t)- b (t, \Theta^{\bar u}_t)  \right) dt 
  +  \E' \int_0^T  p_t \cdot \left( \delta (t, \Theta^u_t)-  \delta (t, \Theta^{\bar u}_t)  \right) dL_t 
  \nonumber \\
&&   \qquad  +  \E' \int_0^T    (X^u_t-X^{\bar u}_t ) \cdot 
 \big(   \nabla_x H(t, \bar u)   + \nabla_y \wh H(t, \bar u)    \big)dt \nonumber\\
 &&   \qquad  + \kappa^{-1}    \E' \int_0^T    (X^u_t-X^{\bar u}_t) \cdot \left(  \1_{\{R_t=0\}}
  \nabla_x \sH(t, \bar u)      +  \1_{\{(R_t)'=0\}}     \nabla_y \wh \sH(t, \bar u) \right)  dt \nonumber\\
 && \qquad + \E' \int_0^t   q_t    \cdot \left( \sigma (t, \Theta^u_t )- \sigma (t, \Theta^{\bar u}_t ) \right)    d L_t  \bigg] .  
\end{eqnarray} 
 This together with  \eqref{J-Jbar} yields 
 \begin{eqnarray}\label{e:7.4}
&& J(u)-J(\bar u)  \nonumber \\
&\geq& \EE \bigg[      \int_0^T \E^\pr \left[ (H(t, u)-H(t, \bar u)) dt + \sH(t, u)- \sH(t, \bar u)) dL_t \right]   \nonumber\\
&&   \qquad  +  \int_0^T   \E'  \left[    (X^u_t-X^{\bar u}_t ) \cdot 
 \big(   \nabla_x H(t, \bar u)   + \nabla_y \wh H(t, \bar u)    \big) \right] dt \nonumber\\
&& \qquad -  \kappa^{-1} \int_0^T \E' \left[ (X^u_t-X^{\bar u}_t )  \cdot \left(  \1_{\{R_t=0\}}
  \nabla_x \sH(t, \bar u)      +  \1_{\{(R_t)'=0\}}     \nabla_y \wh \sH(t, \bar u) \right)  dt  \right]  \bigg], 
   \end{eqnarray}
which  is non-negative by assumption (ii). This proves that $\bar u$ is an optimal control. 
 \qed

 \bigskip

\section{Example}\label{S:8}
\setcounter{equation}{0}
\renewcommand{\theequation}{\thesection.\arabic{equation}}

In this section we given a one-dimensional example with $n=d=1$ and the control domain $U=\R$.
Let $T>0$ and $\lambda \in \R$. Suppose the state equation is of the  following form on $\R$:
\begin{eqnarray}\label{xE1}
d X_t &=& \left(\lambda \EE X_t+X_t  +u_t\right) dt +   dB_{L_t}  \quad \hbox{for } t\in[0,T]
\ \hbox{ with } X_0=x_0\in \R.   
 \end{eqnarray}
  The objective is to minimize the cost functional 
 \begin{eqnarray}\label{JE1}
J(u  )&=&  \frac 1 2 \EE\bigg[\int_{0}^{T}    u_t^2 
       dt + X _T^2  \bigg].
\end{eqnarray}
This corresponds to $b(t, x, y, u)= x+\lambda y +u$, $\delta (t, x, y, u)=0$ and $\sigma (t, x, y, u) =1$ in
\eqref{e:SDE} and $f(t, x, y, u)=u^2/2$, $g(t, x, y, u)=0$  and $h(x, y)=x^2/2$ in \eqref{e:5.8}. 
Thus $H(t, x, y, u, p)=\frac{u^2}2 - (x+\lambda y + u)p$ and 
$\sH (t, x, y, u, p, q)= -q$.

\medskip

In this setting,
the adjoint equation \eqref{BSDE-p} becomes
 \begin{eqnarray}\label{pE1}
\left\{\begin{aligned}
dp_t =& -  ( p_t  +\lambda \EE p_t    )dt  +q_t dB_{L_t},  
    \\
p_T = & - X_T	 .\\
\end{aligned}
\right.
\end{eqnarray}
We look for the solution $(p_t, q_t)$ that are $\bF'$-progressively measurable.
In such a case,  $\bar  p_t = p_t$ and $\bar q_t= q_t$. 
From  Theorem \ref{Necessaryconv}, we then have
\begin{eqnarray} \label{uE1}
u_t=\bar  p_t =p_t.
\end{eqnarray}
For $t\in [0,T]$, we try a solution of (\ref{pE1}) of the form
\begin{eqnarray} \label{solyE1}
p_t=\phi (t)  X_t+ \psi (t) \EE X_t 
\end{eqnarray}
where $\phi$ and $\psi$ are deterministic functions on $[0, T]$ with 
$$
 \phi (T) =-1 \quad \hbox{ and  }   \quad \psi (T)=0.
 $$
 Naturally, differentiating  the above equation leads to
\begin{eqnarray*}
dp_t &=& X_t \phi' (t) dt+ \phi (t) d  X_t+ \psi ' (t) \EE X_t dt +\psi (t) d \EE X_t  \nonumber\\
&=&X_t \phi' (t) dt+ \phi (t)   \big(\lambda \EE X_t+  X_t  +u_t\big )dt+  \phi (t)   dB_{L_t } + \psi ' (t) \EE X_t dt  + \psi (t) \big( (\lambda +1)\EE X_t +\EE u_t \big)dt  \nonumber\\
&=&\big(X_t \phi' (t)  + \phi (t)   \big(\lambda \EE X_t+  X_t  +u_t\big ) + \psi ' (t) \EE X_t  + \psi (t) \big(   (\lambda +1)  \EE X_t +\EE u_t \big)\big) dt
 + \phi (t)   dB_{L_t }.    
\end{eqnarray*}
Comparing the above equation with (\ref{pE1}),    
we have
\begin{eqnarray} \label{e:8.5} 
 X_t \phi' (t)  + \phi (t)   \big(\lambda \EE X_t+  X_t  +u_t\big ) + \psi ' (t) \EE X_t  + \psi (t) \big(   (\lambda +1)  \EE X_t +\EE u_t \big)    =  -  ( p_t +\lambda \EE p_t    ) 
 \end{eqnarray} 
 and
$$
\phi (t)   =q_t .
$$
Using \eqref{uE1} and \eqref{solyE1} in \eqref{e:8.5},  we get 
\begin{eqnarray*}
 && X_t \phi' (t)  + \phi (t)   \big(\lambda \EE X_t+  X_t  +  \phi (t) X_t+ \psi (t) \EE X_t     \big ) + \psi ' (t) \EE X_t  + \psi (t) \big( (\lambda +1)\EE X_t + \phi (t) \EE X_t+ \psi (t) \EE X_t    \big)  \nonumber\\
 && =  -  (\phi (t) X_t+ \psi (t) \EE X_t   + \lambda \phi (t)  \EE  X_t + \lambda  \psi (t) \EE X_t   ) 
 \end{eqnarray*}
By comparing the coefficients of $X_t$ and $\EE X_t$, we obtain 
\begin{equation}\label{e:8.7}
 \phi' (t)   +2 \phi (t)   +  \phi (t) ^2 =  0  
\end{equation}
and
\begin{equation}\label{e:8.8}
2 \lambda \phi (t)   +   2 \phi (t)  \psi (t) + \psi ' (t)  + 2(\lambda +1)\psi (t) +\psi (t)^2=0 .
\end{equation}
It is easy to solve  ODE \eqref{e:8.7}  with terminal condition $\phi (T)=-1$ that 
\begin{equation*}\label{e:8.9}
\phi (t)= - 2/(1+ e^{2(t-T)}) \quad \hbox{for } t\in [0, T].
\end{equation*}
To solve ODE \eqref{e:8.8} with $\psi (T)=0$, set  $y(t):=\psi (T-t)$. Then $y$ satisfies   $y' = f(t, y)$ with 
\begin{equation}\label{e:8.10a}
f(t, y)= y^2  + 2 \left( \lambda + 1 - \frac{2}{1+e^{-2t}} \right) y -\frac{4\lambda} { 1 +e^{-2t}}
\end{equation}
 and  $y(0) = 0$.  Note that $f (t, y)$ is continuous on $\R \times \R$ and is locally Lipschitz continuous in $y$.
 Thus by the Picard-Lindel\"of theorem and the Continuation theorem, 
  there is some $T_0 =T_0(\lambda) \in (0, \infty]$, which we call the positive  explosion time,
  so that    ODE \eqref{e:8.10a} has a unique solution $y$ for $t\in [0, T_0)$ with $y(0)=0$ and it can not be extended
  beyond $T_0$.   Consequently, 
 for $T\in (0, T_0)$, ODE \eqref{e:8.8} has a unique solution $\psi$ on $[0, T]$ with $\psi (T)=0$.

\begin{prop} \label{P:8.1} 
There is some $T_0= T_0 (\lambda) \in (0, \infty]$ so that for every $T \in (0, T_0)$, 
there is a unique optimal control $\bar u\in \cU [0, T]$ for \eqref{xE1}-\eqref{JE1}.
\end{prop} 

\pf    Let $T_0 =T_0(\lambda) \in (0, \infty]$ be the positive explosion time for the solution 
of ODE \eqref{e:8.10a}  with $y(0)=0$. Let $T\in (0, T_0)$, and  
$\phi$ and $\psi$ be the continuous solutions to the ODEs \eqref{e:8.7} and \eqref{e:8.8}
with $\phi (T)=-1$ and $\psi (T)=0$.  By Theorem \ref{exis-unique-SDEs}, there is a unique solution
$X\in \cM [0, T]$ to the following  MF-SDE
\begin{eqnarray}\label{e:8.10}
d \bar X_t &=& \left(  (1+\phi (t)) \bar X_t  + (\lambda + \psi (t) ) \E \bar X_t \right) dt +   dB_{L_t}  \quad \hbox{for } t\in[0,T]
\ \hbox{ with } X_0=x_0\in \R.   
 \end{eqnarray} 
Define 
$$
\bar u_t :=\phi (t) \bar X_t + \psi (t)\E \bar X_t \in \cU' [0, T].
$$
Clearly, $\bar X$ satisfies
$$
d \bar X_t = \left(  \bar X_t  + \lambda   \E \bar X_t + \bar u_t\right) dt +   dB_{L_t}  \quad \hbox{for } t\in[0,T]
\ \hbox{ with } X_0=x_0\in \R.
$$
So we can identify $\bar X$ with $X^{\bar u}$,  the unique $L^2$ strong solution of \eqref{xE1}
with $\bar u$ in place of $u$. Define
$$
p_t =  \phi (t) \bar X_t + \psi (t)\E \bar X_t \quad \hbox{ and } \quad q_t = \phi (t).
$$
Note that $p$ is $\bF'$-progressively measurable and $q$ is deterministic. 
So $\bar p_t :=\E [ p_t | \sF'_t] =p_t$ and $\bar q_t := \E [ q_t | \sF'_t] =q_t$. 
By the calculations that led to \eqref{e:8.7}-\eqref{e:8.8}, we know that $(p, q)$ satisfies 
the adjoint equation \eqref{pE1} with $\bar X_T$ in place of $X_T$ there.
In terms of the notation in \eqref{e:H*}, we have
\begin{eqnarray*}
H^* (t, u) &=&  \frac{u_t^2}2 - (X^u_t  + \lambda (X^u_t)' + u_t) \bar u_t -  \1_{\{R_t=0\}} \kappa^{-1}  q_t,   \\
 \wh H^* (t, u) &=& \frac{((u_t)')^2}2 - ( (X^u_t)'  + \lambda X^u_t + (u_t)') (\bar u_t)' - \1_{\{R_t=0\}} \kappa^{-1}  q_t . 
\end{eqnarray*}
Thus 
\begin{eqnarray} \label{e:8.11}
 && \E  \int_0^T \E^\prime \left[ H ^* (t, u)- H^* (t, \bar u) \right] dt  \nonumber  \\
 &=&  \E  \int_0^T  \left( \frac{u_t^2-\bar u_t^2}2 - (X^u_t -X^{\bar u}_t  + \lambda \E [ X^u_t - X^{\bar u}_t] + u_t-\bar u_t) \bar u_t  \right) dt   \nonumber  \\
 &=&  \E  \int_0^T  \left( \frac{ (u_t-\bar u_t)^2}2 - (X^u_t -X^{\bar u}_t  + \lambda \E [ X^u_t - X^{\bar u}_t] ) \bar u_t  \right) dt    \nonumber \\
&=&  \E  \int_0^T  \left( \frac{ (u_t-\bar u_t)^2}2 - (X^u_t -X^{\bar u}_t ) (1+ \lambda \E [ \bar u_t)   \right) dt   \nonumber  \\
&=& \frac12 \E  \int_0^T  (u_t-\bar u_t)^2dt  
 + \E \Big [ \int_0^T (X^u_t - X^{\bar u}_t)  \cdot  \E' \big[ \nabla_x H^*  (t, \bar u) +\nabla_y \wh H^* (t, \bar u)   \big] dt 
  \nonumber \\
&\geq& \E \Big [ \int_0^T (X^u_t - X^{\bar u}_t)  \cdot  \E' \big[ \nabla_x H^*  (t, \bar u) +\nabla_y \wh H^* (t, \bar u)   \big] dt   \Big] ,
 \end{eqnarray}
 with the equality holds if and only if $u=\bar u$ in $\cU' [0, T]$. 
 Clearly, $h(x, y)=x^2$ is a convex function on $\R\times \R$. Thus by Theorem \ref{Sufficiency},
 $\bar u$ is an optimal control for for \eqref{xE1}-\eqref{JE1}.
 If $u\in \cU'[0, T]$ is another control so that $J(u)=J(\bar u)$, then we have by \eqref{e:7.4}
 that the inequality in \eqref{e:8.11} has to be an equality, that is,   $u$ has to be the same as $\bar u$.
 This establishes the proposition. \qed 
 
 \medskip

  \begin{remark}\rm 
  ODE \eqref{e:8.10a} with initial value $y(0)=0$ is a Riccati equation.
   A more elaborated analysis reveals that its  positive explosion time 
 $T_0 =T_0(\lambda)$ is in fact  infinite for every $\lambda \in \R$.  
 Thus  Proposition \ref{P:8.1} in fact holds with  $T_0= \infty$.  
    But we will not delve into the details here. \qed
  \end{remark}

 \vskip 0.3truein
 
 {\small 
{\bf Shuaiqi Zhang}

\smallskip

     School of Mathematics, China  University of Mining and Technology,    
  Xuzhou,   Jiangsu, 221116, China.  
  
  \smallskip

        Email: \texttt{shuaiqiz@hotmail.com}

\bigskip
	
{\bf Zhen-Qing Chen}

\smallskip

    Department of Mathematics, University of Washington, Seattle,
WA 98195, USA.  

\smallskip

    Email:  \texttt{zqchen@uw.edu}

}
 

\begin{thebibliography}{99}

 \small
   

 
 
 \bibitem {Acciaio}
 B. Acciaio, J. Backhoff-Veraguas and R. Carmona,
  Extended mean-field control problems: stochastic maximum principle and transport perspective, 
  {\it SIAM J. Control Optim. \bf 57}(6) (2019), 3666–3693.
   

\bibitem  {Buckdamn2009a}  
R. Buckdahn, B. Djehiche, J. Li and S. G. Peng.
 Mean-field backward stochastic differential
equations:a limit approach,  {\it  Ann. Probab. \bf 37} (1978), 1524–1565.


\bibitem{Buckdahn2016} 
 R. Buckdahn, J. Li and J. Ma.
  A stochastic maximum principle for general mean-field systems, 
  {\it Appl. Math. Optim. \bf 74} (2016), 507–534.


\bibitem{Buckdahn2009b} 
  R. Buckdahn, J. Li and S. G. Peng.
   Mean-field backward stochastic differential equations
and related partial differential equations, {\it Stoch. Proc. Appl.  \bf 119} (2009), 3081-3834.   
  
 \bibitem{Kac2007}
 M. Kac.
{\it   Foundations of Kinetic Theory,} University of California Press, California, 1956.
 
  
\bibitem{McKean1966}
  H.   McKean.
  A class of Markov processes associated with nonlinear parabolic equations,
{\it Proc. Natl. Acad. Sci. USA. \bf 56}(6) (1966), 1907-1911.
 
 
 
   \bibitem{Meerschaert2004}
 M.  M. Meerschaert and   H.-P. Scheffler.
 Limit theorems for continuous-time random walks with infinite mean waiting times
   {\it J. Appl.  Probab. \bf  41}   (2004),   623-638 
 
 \bibitem{MK}  R. Metzler and J. Klafter. 
  The random walk's guide to anomalous diffusion: A fractional dynamics approach.
  {\it  Phys. Rep. \bf 399(1)} (2002), 1-77. 
 
 
  
\bibitem{Lakhdari2021}
I. E. Lakhdari.
 H. Miloudi and M. Hafayed, Stochastic maximum principle for partially
observed optimal control problems of general McKean–Vlasov differential equations, 
{\it B. Iran. Math. Soc.  \bf 47} (2021), 1021-1043.
 
 
 \bibitem{Lasry2007}
J. M. Lasry and P. L. Lions.
 Mean-field games, 
 {\it Jpn. J. Math. \bf 2} (2007) 229–260.


\bibitem{Li2012} J. Li. 
Stochastic maximum principle in the mean-field controls, 
{\it Automatica \bf 48}(2) (2012), 366-373.
 
 
\bibitem {Wang2022}
  G. C. Wang and Z. Wu.
  A maximum principle for mean-field stochastic control system with
noisy observation, {\it Automatica \bf 137} (2022), 110135.
     
  

\bibitem{ZhangChen2024SICON1} 
 S. Zhang and Z.-Q. Chen.
  Stochastic maximum principle for subdiffusions and its applications, 
  {\it SIAM J. Control Optim. \bf 62} (2024), pp. 953–981. 
 
  
  \bibitem{ZhangChen2024JDE}
 S. Zhang and Z.-Q. Chen.
  Fully coupled forward-backward stochastic differential equations driven by sub-diffusions, 
 {\it J. Differential Equations  \bf 405} (2024), pp. 337–358.  
 
  \bibitem{ZhangChen2024SICON2}
S. Zhang and Z.-Q. Chen.
Stochastic maximum principle for fully coupled forward-backward stochastic differential equations driven by sub-diffusion. 
 { \it SIAM J. Control Optim. \bf 62} (2024), pp. 2433–2455.  

 
 \bibitem{ZC4} 
S. Zhang and Z.-Q. Chen,
 Errata to  {\it Stochastic maximum principle   for sub-diffusions  and its applications}. Preprint 2026.
   


   
\bibitem{Zhang2018} 
  X. Zhang, Z. Y. Sun and J. Xiong.
  A general stochastic maximum principle for a Markov
regime switching jump-diffusion model of mean-field type, 
{\it SIAM J. Control Optim. \bf  56}(4), (2018), 2563-2592.

 

 \end{thebibliography}
\end{document}